\documentclass[arxiv]{agn_article}
\usepackage{agn_all}
\usepackage[babel]{csquotes}

\numberwithin{equation}{section}
\usepackage{graphicx}

\renewcommand{\eval}[2]{#1\big|_{#2}}
\begin{document}
\title{\vspace*{-2.8em}\huge Minimal geodesics on $\GLn$ for left-invariant, right-$\On$-invariant Riemannian metrics}
\author{%
Robert Martin\thanks{Corresponding author, Lehrstuhl f\"{u}r Nichtlineare Analysis und Modellierung, Fakult\"{a}t f\"{u}r Mathematik, Universit\"{a}t Duisburg-Essen,  Thea-Leymann Str.\ 9, 45127 Essen, Germany, email: robert.martin@uni-due.de}
\quad and \quad%
Patrizio Neff\thanks{Head of Lehrstuhl f\"{u}r Nichtlineare Analysis und Modellierung, Fakult\"{a}t f\"{u}r Mathematik, Universit\"{a}t Duisburg-Essen,  Thea-Leymann Str.\ 9, 45127 Essen, Germany, email: patrizio.neff@uni-due.de}
}
\date{\vspace*{-.21em}{\small\today}\vspace*{-1.54em}}
\maketitle
\begin{abstract}
We provide an easy approach to the geodesic distance on the general linear group $\GLn$ for left-invariant Riemannian metrics which are also right-$\On$-invariant. The parametrization of geodesic curves and the global existence of length minimizing geodesics are deduced using simple methods based on the calculus of variations and classical analysis only. The geodesic distance is discussed for some special cases and applications towards the theory of nonlinear elasticity are indicated.
\end{abstract}
{\small
\tableofcontents}

\section{Introduction and preliminaries}
The interpretation of the general linear group $\GLn$ as a Riemannian manifold instead of a simple subset of the linear matrix space $\Rnn$ has recently been motivated by results in the theory of nonlinear elasticity \cite{Neff_Eidel_Osterbrink_2013,neff2015henckymain}, showing a connection between the \emph{logarithmic strain tensor} $\log\sqrt{F^TF}$ of a deformation gradient $F\in\GLpn$ and the \emph{geodesic distance} of $F$ to the special orthogonal group $\SOn$. Since the requirements of objectivity and isotropy strongly suggest a distance measure on $\GLpn$ which is right-invariant under rotations and left-invariant with respect to action of $\GLpn$, we restrict our considerations to Riemannian metrics on $\GLn$ which are left-$\GLn$-invariant as well as right $\On$-invariant.

Although the theory of Lie groups is obviously applicable to $\GLn$ with such a metric, this general approach utilizes many intricate results from the abstract theory of differential geometry and is therefore not easily accessible to readers not sufficiently familiar with these subjects. Furthermore, while the explicit parametrization of geodesic curves has been given for the canonical left-invariant metric on $\GLpn$ \cite{Andruchow2011} as well as for left-invariant, right-$\SOn$-invariant %
metrics on $\SLn$ \cite{Mielke2002}, analogous results are not found in the literature for the more general case on $\GLn$.

The aim of this paper is therefore to provide a more accessible approach to this type of Riemannian metrics on $\GLn$ and the induced geodesic distances as well as to deduce the parametrization of geodesic curves using only basic methods from the calculus of variations and classical analysis. In order to keep this article as self-contained as possible, we will begin by stating (and proving) some very basic facts on Riemannian metrics for the special case of $\GLn$.

\subsection{Basic Definitions}
Let $\Rnn$ denote the set of all $n\times n$ real matrices and let $\id$ denote the identity matrix in $\Rnn$. We define the groups
\begin{alignat*}{2}
	\GLn &= \{X\in\Rnn \setvert \det(X)\neq0\} && \text{(\emph{general linear group})}\,,\\
	\GLpn &= \{X\in\Rnn \setvert \det(X)>0\}\,,\\
	\GLmn &= \{X\in\Rnn \setvert \det(X)<0\}\,,\\
	\SLn &= \{X\in\Rnn \setvert \det(X)=1\} && \text{(\emph{special linear group})}\,,\\
	\On &= \{X\in\Rnn \setvert X^TX=\id\} && \text{(\emph{orthogonal group})}\,,\\
	\SOn &= \{X\in\Rnn \setvert X^TX=\id \,\text{ and } \det(X)=1\} \qquad&& \text{(\emph{special orthogonal group})}\,.
\end{alignat*}
Furthermore, we define the set of \emph{symmetric matrices} $\Symn = \{X\in\Rnn \setvert X^T=X\}$, the set of \emph{positive definite symmetric matrices} $\PSymn = \{X\in\Rnn \setvert X^T=X \text{ and } v^T Xv>0\;\;\forall\,v\in\R^n\setminus\{0\}\}$ and the set of \emph{skew symmetric matrices} $\son = \{X\in\Rnn \setvert X^T=-X\}$.

\subsection{Distance functions on $\Rnn$}
A \emph{distance} on a set $M$ is a function $\dist\colon M \times M \to [0,\infty]$ with
\begin{align}
	&\dist(A,A) = 0\,, \qquad \dist(A,B) = \dist(B,A) \nnl
	\text{and} \quad &\dist(A,C) \leq \dist(A,B) + \dist(B,C) &&\text{(\emph{triangle inequality})}	\label{eq:distanceTriangleInequality}
\end{align}
for all $A,B,C \in M$. Note that $\dist$ is a \emph{metric} on $\Rnn$ if and only if, additionally,
\[
	0 < \dist(A,B) < \infty \quad \text{for all $A \neq B$.}
\]
A common distance on $\Rnn$ is the \emph{Euclidean distance}: The canonical inner product
\begin{align}
	\eucproduct{M,N} = \tr(M^T N) = \sum_{ij=1}^{n} M_{ij} N_{ij}\,, \quad M,N \in \Rnn\,,
\end{align}
where $\tr M = \sum_{i=1}^n M_{i,i}$ denotes the \emph{trace} of $M$, induces the Euclidean norm (or \emph{Frobenius matrix norm})
\begin{align}
	\norm{M} = \sqrt{\eucproduct{M,M}} = \sqrt{\smash{\sum_{i,j=1}^{n}}\vphantom{\sum^n} M_{ij}^2}\;,
\end{align}
and the Euclidean distance is the metric given by
\begin{align}
	\disteuc(A,B) = \norm{A-B}\,.
\end{align}

While $\disteuc$ induces a distance function on $\GLn$ as well, it does not appear as a \enquote{natural} inner property of the general linear group: since $\GLn$ is not a linear space, the term $A-B$ depends on the underlying algebraic structure of the vector space $\Rnn$. Furthermore, because $\GLn$ is not a closed subset of $\Rnn$, it is not complete with respect to the Euclidean distance.

A more proper distance measure should take into account the algebraic properties of $\GLn$ as a group. To find such a function we interpret $\GLn$ as a Riemannian manifold: as an open subset of $\Rnn$, the tangent space $T_A \GLn$ at an arbitrary point $A \in \GLn$ is given by\footnote{Note that in the theory of Lie Groups $\gln$ usually denotes (or, more precisely, is identified with) the tangent space $T_\id \GLn$ at the identity only. From the perspective of classical analysis, all tangent spaces $T_A\GLn$ can simply be identified with $\Rnn$, allowing us to employ a much simpler notion of \enquote{smoothness}.}
\begin{align}
	T_A \GLn = A\cdot T_\id \GLn \cong A\cdot \Rnn = \Rnn \equalscolon \gln\,.
\end{align}
To obtain a distance function respecting this structure on $\GLn$, we will consider a measurement along connecting curves.

\subsection{The geodesic distance}

A \emph{Riemannian metric} on $\GLn$ is a smooth (i.e.\ infinitely differentiable) function
\begin{align}
	g\colon \GLn \times \gln \times \gln \to \R, \; (A,M,N) \mapsto g_A(M,N)
\end{align}
such that for every fixed $A \in \GLn$ the function $g_A(\cdot,\cdot)\colon \gln \times \gln \to \R$ is a positive definite symmetric bilinear form, i.e.
\begin{align*}
	g_A(\lambda_1 M_1 + \lambda_2 M_2, N) &= \lambda_1 g_A(M_1, N) + \lambda_2 g_A(M_2, N)\,, \quad
	g_A(M,N) = g_A (N,M)\,, \quad
	g_A(T,T) > 0
\end{align*}
for all $M,N,M_1,M_2,T \in \gln$, $\lambda_1, \lambda_2 \in \R$, $T \neq 0$.
A Riemannian metric allows for a measurement of sufficiently smooth curves in $\GLn$: the \emph{length} $\len$ and \emph{energy} $\ener$ of a curve $X \in C^1([a,b]; \GLn$) are given by
\begin{equation}
	\len(X) \colonequals \int\nolimits_0^1 \sqrt{g_{X(t)}(\Xdot(t), \Xdot(t))} \:\dt \qquad\text{and}\qquad \ener(X) \colonequals \int\nolimits_0^1 g_{X(t)}(\Xdot(t), \Xdot(t)) \:\dt\,,
\end{equation}
where we employ the notation $\Xdot(t) = \ddt X(t)$.\\
Note that the length of $X$ is defined similarly to the length of curves in Euclidean spaces. Thus many well-known properties, like invariance under reparameterization, still hold in the Riemannian case. Some such properties will be discussed further in the following section.

The geodesic distance $\dg(A,B)$ between $A,B\in\GLn$ can now be defined as the infimum over the length of curves connecting $A$ and $B$. For this we need an exact definition of the admissible sets of curves.

\begin{definitions}
\label{def:admissibleSets}
	We denote by
	\begin{align}
		C^k_r([a,b];\GLn) \colonequals \{X \in C^k([a,b];\GLn) \setvert \Xdot(t) \neq 0 \;\; \forall\, t \in [a,b]\}
	\end{align}
	the set of \emph{regular $k$-times differentiable} curves in $\GLn$ over the interval $[a,b]$. Note that, by the usual definition of differentiability on closed intervals as the restriction of differentiable functions on $\R$, the $i$-th derivative $X^{(i)}(a)$ and $X^{(i)}(b)$ at the boundaries is well-defined via the one-sided limits $\underset{t\searrow a}{\lim}$ $X^{(i)}(t)$ and $\underset{t\nearrow b}{\lim}$ $X^{(i)}(t)$. We now define the set of \emph{piecewise $k$-times differentiable curves} in $\GLn$ over the interval $[a,b]$ by
	\begin{align}
		\adm^k([a,b]) \colonequals \big\{\;X \in C^0(&[a,b];\GL(n)) \setvert \exists\,a = a_0 < a_1 < \dots < a_{m+1} = b, \; \forall\,j = 0,\dotsc,m:\nonumber\\ &\eval{X}{[a_j,a_{j+1}]} \in C^k_r([a_j,a_{j+1}];\GLn) \;\big\}\,.
	\end{align}
	Note that, by this definition, partially differentiable curves are continuous everywhere.
	Finally, for $A,B \in \GLn$, the \emph{admissible set of curves connecting $A$ and $B$} is
	\begin{equation}
		\adsetAB \colonequals \{ X\in\adm^1([a,b]) \setvert a,b\in\R\,, \, X(a)=A\,, \, X(b)=B \}\,,
	\end{equation}
	the admissible set over the fixed interval $[a,b]$ is denoted by
	\begin{equation}
		\adsetAB([a,b]) \colonequals \{ X\colon [a,b] \to \GLn, X \in \adsetAB\}
	\end{equation}
	and the general \emph{admissible set of curves} is
	\begin{equation}
		\adm \colonequals \bigcup_{A,B\in\GLn} \adsetAB\,.
	\end{equation}
\end{definitions}

\begin{remark}
	While the notion of piecewise differentiability is often found in the literature, a specific definition is sometimes omitted. The definition used here guarantees the existence of one-sided limits $\underset{t\searrow t_0}{\lim} X^{(i)}(t)$, $\underset{t\nearrow t_0}{\lim} X^{(i)}(t)$ everywhere and thus, in particular, that the length $\len(X) < \infty$ is well-defined.
\end{remark}
We can now properly define the geodesic distance function:
\begin{definition}
	Let $A,B \in \GLn$. Then
	\begin{align}
		\dg(A,B) \colonequals \underset{X\in\adsetAB}{\inf} \len(X)
	\end{align}
	is called the \emph{geodesic distance} between two matrices $A$ and $B$.
\end{definition}

\begin{remark}
	It is easy to verify that $\dg$ is indeed a distance function; in order to see that it satisfies the triangle inequality \eqref{eq:distanceTriangleInequality}, choose a curve $X \in \adsetAB$ connecting $A$ and $B$ with $\len(X) \leq \dg(A,B) + \eps$ as well as a curve $Y \in \adset{B}{C}$ connecting $B$ and $C$ with $\len(Y) \leq \dg(B,C) + \eps$. We assume (without loss of generality, as we will see in Lemma \ref{lemma:constantSpeed}) that both $X$ and $Y$ are defined on the interval $[0,1]$. Let $Z$ denote the curve obtained by \enquote{attaching} $Y$ to $X$, i.e.
	\begin{equation}
		Z\colon [0,2] \to \GLn, \quad Z(t) = \begin{cases}X(t)&: t\in[0,1]\\ Y(t-1)&: t\in[1,2]\end{cases}\,.
	\end{equation}
	Then $Z$ is piecewise differentiable, $Z(0)=X(0)=A$ and $Z(2)=Y(1)=C$, and thus $Z\in\adset{A}{C}$. We find\footnote{Here and throughout we will often omit the integration variable and write e.g.\ $\Zdot$ instead of $\Zdot(t)$.}
	\begin{align}
		\len(Z) = \int\nolimits_0^2 \sqrt{\fullg{Z}(\Zdot,\Zdot)} \:\dt &= \int\nolimits_0^1 \sqrt{\fullg{Z}(\Zdot,\Zdot)} \:\dt + \int\nolimits_1^2 \sqrt{\fullg{Z}(\Zdot,\Zdot)} \:\dt \\
		&= \int\nolimits_0^1 \sqrt{\fullg{X}(\Xdot,\Xdot)} \:\dt + \int\nolimits_0^1 \sqrt{\fullg{Y}(\Ydot,\Ydot)} \:\dt \nonumber \\
		&= \len(X) + \len(Y) \leq \dg(A,B) + \dg(B,C) + 2\eps \nonumber
	\end{align}
	and thus $\dg(A,C) \leq \len(Z) \leq \dg(A,B) + \dg(B,C) + 2\eps$ for all $\eps > 0$, which shows the triangle inequality.
\end{remark}

The geodesic distance can be considered a generalization of the Euclidean distance: if we measure the length of a curve $\gamma\colon [a,b] \to \Rnn$ by $\len(\gamma) = \int\nolimits_a^b \sqrt{\eucproduct{\gammadot,\gammadot}} \,\dt$, then the shortest curve connecting $A,B \in \Rnn$ is a straight line with length $\norm{A-B}$. Thus the Euclidean distance can be interpreted as the infimum over the length of connecting curves as well.
	
Furthermore, $\GLn$ is not connected, but can be decomposed into two connected components $\GLpn = \{A\in\GLn \setvert \det A > 0\}$ and $\GL^-(n) = \{A\in\GLn \setvert \det A < 0\}$. Thus $\dg(A,B) < \infty$ if and only if $A$ and $B$ are in the same connected component, i.e.\ iff $\det (AB) > 0$. As we will see later on, for left-invariant Riemannian metrics we can focus on the case $A,B \in \GLpn$ without loss of generality.

\subsection{Length and energy of curves}
In order to further investigate the geodesic distance, some basic properties of curves in $\GLn$, the length and the energy functional are required. Some of these properties can be found in any textbook on differential geometry; however, in order to keep this article self contained and accessible to readers unfamiliar with the methods of general differential geometry, we explicitly state and prove them here. Many properties of curves in $\GLn$ also correspond directly to the case of curves in the Euclidean space (see e.g.\ \cite[Chapter 12]{Koenigsberger2004}).

Consider the integrand $\sqrt{\fullg{X(t)}(\Xdot(t),\Xdot(t))}$ in the definition of the length of a curve $X$. In analogy to the Euclidean space, we call this term the \emph{speed} of $X$. Since, by definition of a Riemannian metric, $\fullg{A}(\cdot,\cdot)$ defines an inner product on the tangent space $T_{A}\GLn = \Rnn = \gln$, it induces a norm on $\gln$. We will therefore simplify notation by writing $\norm{M}_A = \sqrt{\fullg{A}(M,M)}$ for $A\in\GLn$, $M\in\gln$.
Then a differentiable curve $X \in \adm([a,b])$ has \emph{constant speed} with regard to the Riemannian metric $g$ if the mapping $t\mapsto\norm{\Xdot(t)}_{X(t)} = \sqrt{\fullg{X(t)}(\Xdot(t),\Xdot(t))}$ is constant on $[a,b]$. Furthermore, if $X$ is only piecewise differentiable, we say that $X$ has constant speed if there exists $c>0$ such that $\norm{\Xdot(t)}_{X(t)}=c$ for all $t\in(a,b)$ at which $X$ is differentiable.

\begin{lemma}
\label{lemma:constantSpeed}
	Let $X \in \adm^1([a,b])$. Then there exists a unique piecewise differentiable $\varphi\in C^0([a,b];[a,b])$, $\varphi(a)=a$, $\varphi(b)=b$, $\varphi'(t)>0$ such that $X \circ \varphi$ has constant speed.
\end{lemma}
\begin{proof}
	See Lemma \ref{lemma:constantSpeedAppendix} in the appendix.
\end{proof}

For $X \in \adm([a,b])$, we denote by
\[
	X_c \colonequals X \circ \varphi \: \in\adm([a,b])
\]
the (unique) reparametrization from Lemma \ref{lemma:constantSpeed}.

An important property of the length functional is its invariance under reparametrizations:
\begin{lemma}
	\label{lemma:lengthInvariance}
	Let $X \in \adm([a,b])$, and let $\varphi \in  C^0([c,d];[a,b])$ be a piecewise continuously differentiable function with $\varphi(c)=a$, $\varphi(d)=b$ and $\varphi'(t)>0$. Then
	\begin{equation}
		\len(X \circ \varphi) = \len(X)\,.
	\end{equation}
\end{lemma}
\begin{proof}
	See Lemma \ref{lemma:lengthInvarianceAppendix} in the appendix.
\end{proof}
To explicitly compute the geodesic distance, we will primarily search for \emph{length minimizers}, i.e.\ curves $X \in \adsetAB$ which satisfy
\begin{align}
	 \len(X) = \underset{Y\in\adsetAB}{\inf} \len(Y) = \dg(A,B)\,.
\end{align}
Since every curve on $[a,b]$ can be reparametrized (by scaling and shifting) to a curve of the same length on an arbitrary interval $[c,d]$, a restriction of the admissible interval does not change the infimal length:
\begin{align}
	\underset{Y\in\adsetAB}{\inf} \len(Y) = \underset{Y\in\adsetAB([a,b])}{\inf} \len(Y)\,;
\end{align}
recall from Definition \ref{def:admissibleSets} that $\adsetAB([a,b])$ denotes the set of piecewise differentiable curves defined on $[a,b]$ connecting $A$ and $B$. Furthermore, if there exists a length minimizer in $\adsetAB$, then there exists one in $\adsetAB([a,b])$ as well. If we are interested only in the length of a curve $X$, we will therefore often assume without loss of generality that $X$ is defined on $[0,1]$ and that $X$ has constant speed.

However, the energy functional is \emph{not} invariant under reparameterization: For a given curve $X\colon[0,1] \to \GLn$ and $\lambda > 0$, the energy of the curve $Y\colon[0,\lambda] \to \GLn, Y(t) = X(\frac{t}{\lambda})$ is
\begin{align}
	\ener(Y) = \int\nolimits_0^\lambda \fullg{Y(t)}(\Ydot(t),\Ydot(t)) \:\dt &= \int\nolimits_0^\lambda \fullg{X(\frac{t}{\lambda})}(\tfrac{1}{\lambda}\Xdot(\tfrac{t}{\lambda}), \tfrac{1}{\lambda}\Xdot(\tfrac{t}{\lambda})) \:\dt \nonumber \\
	&= \frac{1}{\lambda}\int\nolimits_0^\lambda \frac{1}{\lambda} \fullg{X(\frac{t}{\lambda})}(\Xdot(\tfrac{t}{\lambda}), \Xdot(\tfrac{t}{\lambda})) \:\dt \nonumber \\
	&= \frac{1}{\lambda}\int\nolimits_0^1 \fullg{X(t)}(\Xdot(t), \Xdot(t)) \:\dt \;=\; \frac{1}{\lambda}\,\ener(X)\,.
\end{align}
For $\lambda \to \infty$, we see that by admitting arbitrary parametrizations, the infimal energy of curves connecting $A$, $B$ is zero whenever $A$ and $B$ can be connected:
\begin{align}
	\underset{Y\in\adsetAB}{\inf} \ener(Y) = 0 \qquad \forall\, A,B \in \GLn: \det(AB) > 0\,.
\end{align}
We will therefore call $X\colon[a,b]\to\GLn, X\in\adsetAB([a,b])$ an \emph{energy minimizer} if and only if it minimizes the energy over all curves \emph{over the same parameter interval} connecting $A$ and $B$:
\begin{align}
	X\in\adsetAB([a,b]) \quad\text{ and }\quad \ener(X) = \underset{Y\in\adsetAB([a,b])}{\inf} \ener(Y)\,.
\end{align}
The next two lemmas show important relations between length minimizers and energy minimizers.

\begin{lemma}
	\label{lemma:lengthEnergyInequality}
	Let $X \in \adset{}{}([a,b])$ be a piecewise differentiable curve. Then
	\begin{align}
		\len(X)\leq \sqrt{b-a}\sqrt{\ener(X)}\,, \label{eq:energyLengthEstimate}
	\end{align}
	and equality holds if and only if $X$ has constant speed.
\end{lemma}
\begin{proof}
	\begin{align}
		[\len(X)]^2 = \left( \int\nolimits_a^b \norm{\Xdot(t)}_{X(t)} \,\dt\right)^2
		&= \left( \int\nolimits_a^b 1 \cdot \norm{\Xdot(t)}_{X(t)} \,\dt \right)^2\nonumber \\
		&\leq \int\nolimits_a^b 1 \:\dt \:\cdot\: \int\nolimits_a^b \norm{\Xdot(t)}_{X(t)}^2 \:\dt \;=\; (b-a)\ener({\gamma})\,. \label{eq:hoelder}
	\end{align}
	The inequality in ($\ref{eq:hoelder}$) is due to the Hölder inequality, and equality holds if and only if $\norm{\Xdot}_{X}$ and $1$ are linearly dependent in $L^2([a,b])$, meaning $\norm{\Xdot}_X \equiv \text{constant}$ on $[a,b]$.
\end{proof}

\begin{lemma}
\label{lemma:energyLengthMinimizerRelation}
	For $X \in \adsetAB([a,b])$, the following are equivalent:
	\begin{align*}
		\text{i) }\ener(X) &= \minYab \ener(Y)\,, \\
		\text{ii) }\len(X) &= \minYab \len(Y) = \minY \len(Y)\text{ and $X$ has constant speed.}
	\end{align*}
\end{lemma}
\begin{proof}
	Recall that $X_c$ denotes the (unique) parameterization of $X$ with constant speed on $[a,b]$. The invariance of the length under reparameterization implies $\len(X_c) = \len(X)$, and inequality \eqref{eq:energyLengthEstimate} from the previous lemma yields
	\begin{equation}
		\ener(X) \geq \frac{1}{b-a} [\len(X)]^2 = \frac{1}{b-a} [\len(X_c)]^2 = \ener(X_c)\,,
	\end{equation}
	where, again, equality holds if and only if $X$ has constant speed, i.e.\ $X = X_c$. Therefore, every minimizer of $\ener$ must have constant speed, and it remains to show that if $X$ has constant speed, then the equivalence
	\begin{equation}
		\ener(X) = \minYab \ener(Y) \iff \len(X) = \minYab \len(Y)
	\end{equation}
	holds. If we assume that $X=X_c$, then
	\begin{align}
		\ener(X) &= \minYab \ener(Y) \\
		\Rightarrow \len(Y) &= \len(Y_c) = \sqrt{b-a}\sqrt{\ener(Y_c)} \nonumber \geq \sqrt{b-a}\sqrt{\ener(X)} = \len(X) \qquad \forall\, Y\in\adsetAB([a,b]) \nonumber \\
		\Rightarrow \len(X) &= \minYab \len(Y)\,,\nonumber
	\intertext{as well as}
		\len(X) &= \minYab \len(Y) \\
		\Rightarrow \ener(Y) &\geq \frac{1}{b-a} [\len(Y)]^2 \geq \frac{1}{b-a}[\len(X)]^2 = \ener(X) \qquad \forall\, Y\in\adsetAB([a,b]) \nonumber \\
		\Rightarrow \ener(X) &= \minYab \ener(Y)\,,\nonumber
	\end{align}
	which concludes the proof.
\end{proof}

\subsection{Left-invariant, right-$\On$-invariant Riemannian metrics}

In the following, we will only consider Riemannian metrics that are left-invariant as well as right-invariant under $\On$.

\begin{definition}
A Riemannian metric $g$ on $\GLn$ is called \emph{left $\GLn$-invariant} (or simply \emph{left-invariant}) if
\begin{align}
	\fullg{CA}(CM,CN) = \fullg{A}(M,N)
\end{align}
for all $A,C \in \GLn$ and $M,N \in \gln$.
\end{definition}

The left-invariance of a Riemannian metric $g$ can be applied directly to the geodesic distance: let $A,B,C \in \GLn$. Then for every given curve $X \in \adsetAB$ connecting $A$ and $B$ we can define a curve $Y = CX \in \adset{(CA)}{(CB)}$ connecting $CA$ and $CB$ by $Y(t) = C X(t)$. We assume without loss of generality that $X$ is defined on the interval $[0,1]$ and find
\begin{align}
	\len(Y) &= \int\nolimits_0^1 \sqrt{\fullg{Y(t)}(\Ydot(t), \Ydot(t))} \: \dt \\
	&= \int\nolimits_0^1 \sqrt{\fullg{CX(t)}(C\Xdot(t), C\Xdot(t))} \: \dt \nonumber  = \int\nolimits_0^1 \sqrt{\fullg{X(t)}(\Xdot(t), \Xdot(t))} \: \dt  = \len(X)\,. \nonumber
\end{align}
Analogously, for every curve $Y$ connecting $CA$ and $CB$, the curve $X=C^{-1}Y$ connects $A$ and $B$ with $\len(X)$ = $\len(Y)$. Thus for every curve connecting $A$ and $B$, we can find a curve of equal length connecting $CA$ and $CB$ and vice versa. Therefore
\begin{equation}
	\dg(CA,CB) = \underset{Y \in \adset{(CA)}{(CB)}}{\inf} \len(Y) = \underset{X \in \adset{A}{B}}{\inf} \len(X) = \dg(A,B)\,.
\end{equation}
In particular, this left-invariance of the geodesic distance implies
\begin{align}
	\dg(A,B) &= \dg(\id,A^{-1}B)\,.
\end{align}
We will therefore often focus on the case $A = \id$. Since $A\inv B\in\GLpn$ for $A,B\in\GL^-(n)$, the geodesic distance on $\GL^-(n)$ is completely determined by the geodesic distance on $\GLpn$ for left-$\GLn$-invariant metrics.

\begin{definition}
A Riemannian metric $g$ on $\GLn$ is called \emph{right-$\On$-invariant} if
\begin{align}
	\fullg{AQ}(MQ,NQ) = \fullg{A}(M,N)
\end{align}
for all $A \in \GLn$, $M,N \in \gln$ and $Q \in \On$.
\end{definition}

Such invariant metrics appear in the theory of elasticity, where right-$\On$-invariance follows from material isotropy, while objectivity implies the left-invariance. We will show that a Riemannian metric satisfying both invariances is uniquely determined up to three parameters $\mu,\mu_c,\kappa > 0$ and given by an \emph{isotropic inner product} $\fullproduct{\cdot , \cdot}$ of the form
\begin{align}
	\fullproduct{M,N} \ratio&= \mu \innerproduct{\dev\sym M,\dev\sym N} + \mu_c \innerproduct{\skew M,\skew N} + \frac{\kappa}{n} (\tr M) (\tr N) \label{eq:definitionFullProduct}\\
		&= \mu \tr[(\dev\sym M)^T \dev\sym N] + \mu_c \tr[(\skew M)^T \skew N] + \frac{\kappa}{n} (\tr M) (\tr N)\,, \nonumber
\end{align}
where $\sym M = \half(M + M^T)$ denotes the symmetric part, $\skew M = \half(M - M^T)$ is the skew symmetric part and $\dev M = M - \frac{\tr M}{n} \id$ denotes the deviator of $M$. Here and throughout, $\innerproduct{M,N} = \tr(M^TN)$ denotes the canonical inner product on $\gln$.

We will state some basic properties of $\fullproduct{\cdot , \cdot}$. The proof can be found in the appendix (Lemma \ref{lemma:isotropicProductLemmaAppendix}).

\begin{lemma}
	Let $M,N \in \gln$. Then
	\begin{alignat}{2}
		&\text{i)} &&\fullproduct{Q^TMQ, Q^TNQ} = \fullproduct{M,N} \quad \text{ for all } Q \in \On\,, \label{eq:isotropicProperty} \\
		&\text{ii)} &&\fullproduct{M, N} = \mu \innerproduct{\dev\sym M, N} + \mu_c \innerproduct{\skew M, N} + \frac{\kappa}{n} \tr M \tr N\,, \nonumber \\
		&&& \hphantom{\fullproduct{M, N}} = \innerproduct{\mu \dev\sym M + \mu_c \skew M + \frac{\kappa}{n} \tr (M) \cdot \id\,,\,N}  \label{eq:oneSidedIsoProduct}\\
		&\text{iii)} &&\innerproduct{M, N}_{1,1,1} = \innerproduct{M,N}\,, \label{eq:canonicalInnerProductSpecialCaseOfIsoProduct}\\
		&\text{iv)}\quad &&\fullproduct{S, W} = 0 \quad \text{ for all } S \in \Symn,\: W \in \son\,,
	\end{alignat}
	where $\Symn$ and $\son$ denote the sets of symmetric and skew symmetric matrices in $\Rnn$ respectively.
\end{lemma}
In particular, \eqref{eq:canonicalInnerProductSpecialCaseOfIsoProduct} implies that the canonical inner product can be interpreted as a special case of $\fullproduct{\cdot , \cdot}$ with $\mu = \mu_c = \kappa = 1$.
Furthermore we will denote by
\begin{equation}
	\fullnorm{M} = \sqrt{\fullproduct{M,M}} = \sqrt{\mu\norm{\dev\sym M}^2 + \muc\norm{\skew M}^2 + \frac{\kappa}{n}\tr(M)^2}
\end{equation}
the norm induced by $\fullproduct{\cdot , \cdot}$. Here and throughout, $\norm{M} = \sum_{i,j=1}^n M_{ij}^2$ denotes the canonical matrix norm.

We can now show the connection between the isotropic inner product and left-invariant, right-$\On$-invariant Riemannian metrics.

\begin{proposition}
	~\\
	i) A Riemannian metric $g$ on $\GLn$ is left-invariant if and only if $g$ is of the form
	\begin{equation}
		\fullg{A}(M,N) = \innerproduct{A^{-1}M,A^{-1}N}_\ast\,,
	\end{equation}
	where $\innerproduct{\cdot,\cdot}_\ast$ is an inner product on the tangent space $\gln = T_\id \GLn$ at the identity $\id$.
	\\[2mm]
	ii) A left-invariant metric $g$ is additionally right-$\On$-invariant if and only if $g$ is of the form
	\begin{equation}
		\fullg{A}(M,N) = \fullproduct{A^{-1}M,A^{-1}N} \label{eq:leftinvRightSOinvMetric}
	\end{equation}
	with $\mu, \mu_c, \kappa > 0$.
\end{proposition}

\begin{proof}
	i) Let $\innerproduct{\cdot , \cdot}_\ast$ be an inner product on $\gln$ and $g$ be defined by $\fullg{A}(M,N) = \innerproduct{A^{-1}M,A^{-1}N}_\ast$. Then we find
	\begin{align}
		\fullg{CA}(CM,CN) &= \innerproduct{(CA)^{-1}CM,(CA)^{-1}CN}_\ast \\
		&= \innerproduct{A^{-1}C^{-1}CM,A^{-1}C^{-1}CN}_\ast \nonumber \;=\; \innerproduct{A^{-1}M,A^{-1}N}_\ast \;=\; g_A(M,N)\,, \nonumber 
	\end{align}
	hence $g$ is left-invariant.
	
	Now let $g$ be an arbitrary left-invariant Riemannian metric on $\GLn$. Then
	\begin{equation}
		\innerproduct{M,N}_\ast \colonequals \fullg{\id}(M,N)
	\end{equation}
	defines an inner product on $\gln$, and we find
	\begin{align}
		g_A(M,N) &= \fullg{A\cdot\id}(AA^{-1}M,\,AA^{-1}N) = g_\id(A^{-1}M,\,A^{-1}N) = \innerproduct{A^{-1}M,\,A^{-1}N}_\ast\,. \nonumber
	\end{align}
	ii) If $g$ is defined by $g_A(M,N) = \fullproduct{A^{-1}M, A^{-1}N}$, then $g$ is left-invariant according to i). Let $Q \in \On$. Then
	\begin{align}
		\fullg{AQ}(MQ,NQ) &= \fullproduct{(AQ)^{-1}MQ, (AQ)^{-1}NQ} = \fullproduct{Q^T A^{-1} MQ, Q^T A^{-1} NQ}\,. \nonumber
	\end{align}
	We apply the isotropy property \eqref{eq:isotropicProperty} to find
	\begin{equation}
		\fullproduct{Q^T A^{-1} MQ, Q^T A^{-1} NQ} = \fullproduct{A^{-1} M, A^{-1} N} = g_A(M,N)\,,
	\end{equation}
	which implies the right-$\On$-invariance of $g$.
	
	Finally, let $g$ be an arbitrary left-invariant, right-$\On$-invariant metric on $\GLn$. Again we define the inner product $\innerproduct{\cdot , \cdot}_\ast$ by $\innerproduct{M,N}_\ast \colonequals \fullg{\id}(M,N)$. Then for every $Q\in\On$
	\begin{align}
		\innerproduct{Q^T MQ, Q^T NQ}_\ast &= \fullg{\id}(Q^T MQ, Q^T NQ) \label{eq:isotropyOfTheInnerProduct}\\
		&= \fullg{Q \id}(MQ, NQ) = \fullg{\id Q}(MQ, NQ) = \fullg{\id}(M, N) = \innerproduct{M,N}_\ast\,. \nonumber
	\end{align}
	According to a well-known representation formula for isotropic linear mappings on $\Rnn$ \cite{Boor1985}, this invariance directly implies that $\innerproduct{\cdot , \cdot}_\ast$ has the desired form \eqref{eq:definitionFullProduct}.
\end{proof}

In the following sections $g$ will denote a left-invariant, right-$\On$-invariant Riemannian metric as given in \eqref{eq:leftinvRightSOinvMetric}, unless otherwise indicated.

\section{Energy minimizers and the calculus of variations}
\label{section:calculusOfVariations}
In the theory of Lie Groups and, more generally, Riemannian geometry, it can be shown that every length minimizing curve on a sufficiently smooth Riemannian manifold is a geodesic \cite[Corollary 3.9]{docarmo1992}. However, by focussing solely on the considered special case, we can circumvent the methods of abstract Riemannian geometry (such as the Levi-Civita-connection) and find a differential equation characterizing the minimizing curves by a straightforward application of the classical calculus of variations. A similar approach can be found in \cite{LeeCKP2007}, where a geodesic equation similar to \eqref{eq:canonicalDifferentialEquation} is computed for the \emph{right-invariant} Riemannian metric induced by the canonical inner product on $\gln$.\\
In preparation we need the following two lemmas. For any continuously differentiable function $f\in C^1(\GLn;\GLn)$, we denote by $Df[A]$ the total derivative of $f$ at $A\in\GLn$, and $Df[A].H\in\gln$ denotes its directional derivative at $A$ in direction $H$.

\begin{lemma}
\label{lemma:invDerivative}
	Define $f\in C^1(\GLn,\GLn)$ by $f(X)=X\inv$. Then
	\begin{equation}
		Df[A].H = -A\inv H A\inv\,.
	\end{equation}
\end{lemma}
\begin{proof}
	A short proof can be found in \cite{NeffCISMnotes}.
\end{proof}

\begin{lemma}
Let $M,N,P\in\gln$. Then
	\begin{align}
		\innerproduct{M,PN} = \innerproduct{P^TM,N} \qquad\text{and}\qquad
		\innerproduct{M,NP} = \innerproduct{MP^T,N}\,.
	\end{align}
\end{lemma}
\begin{proof}
Direct computation yields
\begin{align*}
	\innerproduct{M,PN} &= \tr((PN)^TM) = \tr(N^T(P^TM)) = \innerproduct{P^TM,N}\,,\\
	\innerproduct{M,NP} &= \tr((NP)^TM) = \tr(P^TN^TM) = \tr(N^T(MP^T)) = \innerproduct{MP^T,N}\,.\qedhere
\end{align*}
\end{proof}

\subsection{The geodesic equation for the canonical inner product}
\label{subsection:canonicalDifferentialEquation}
First, assume that $\fullproduct{\cdot,\cdot}$ is the canonical inner product on $\gln=\Rnn$, i.e.\ that $\mu=\mu_c=\kappa=1$. Let $X\in\adsetAB\,\cap C^2([0,1];\GLp(n))$ be a two-times differentiable length minimizing curve from $[0,1]$ to $\GLp(n)$, where the length is measured by the left-invariant Riemannian metric $g$ with
\begin{align}
	g_Z(M,N) &\colonequals \innerproduct{Z^{-1}M, Z^{-1}N} = \tr((Z^{-1}M)^TZ^{-1}N)\nonumber\,,
\end{align}
and assume without loss of generality (Lemma \ref{lemma:constantSpeed}) that X has constant speed. Then it follows from Lemma \ref{lemma:energyLengthMinimizerRelation} that X also minimizes the energy functional $E$. We will characterize the minimizer $X$ by the Euler–Lagrange equation corresponding to the energy, which we will call the \emph{geodesic equation}, in consistence with a more general differential geometric definition of geodesics \cite[Definition 1.4.2]{Jost1998}. Let
\begin{equation}
	C^1_0([0,1];\gln) \colonequals \{\deltaX\in C^1([0,1];\gln) \setvert \supp(\deltaX)\subset(0,1)\}
\end{equation}
denote the set of \emph{variations} in $\gln$ with compact support. For a fixed variation $\deltaX\in C^1_0([0,1];\gln)$, define $Z_h \colonequals X + h\, \deltaX$ for $h\in (-\eps, \eps)$. For sufficiently small $\eps$, we have $Z_h(t)\in \GLp(n)$ for all $t\in[0,1]$ and therefore $Z_h\in \adsetAB$ (note that $\deltaX(0)=\deltaX(1)=0$ and thus $Z_h(0)=X(0)=A, Z_h(1)=B$). The minimizing property of X implies the stationarity condition
\begin{align}
	0 = \frac{1}{2}\evalhzero{\ddh \ener(X+h\,\deltaX)} &= \frac{1}{2}\ddh\int\nolimits_0^1 \evalhzero{\innerproduct{{Z_h}^{-1} \dot{Z}_h, {Z_h}^{-1} \dot{Z}_h} \,\dt}\nonumber\\
	&= \frac{1}{2} \int\nolimits_0^1 \ddh \evalhzero{\innerproduct{{Z_h}^{-1} \dot{Z}_h, {Z_h}^{-1} \dot{Z}_h}} \dt\nonumber\\
	&= \int\nolimits_0^1 \,\innerproduct{\evalhzero{{Z_h}^{-1} \dot{Z}_h}, \ddh\evalhzero{{Z_h}^{-1} \dot{Z}_h}} \,\dt\,. \label{eq:localVariationalCalcMiddle}
\end{align}
Since $\evalhzero{{Z_h}^{-1}}=X\inv$ and $\evalhzero{\dot{Z}_h}=\Xdot$, we can use Lemma \ref{lemma:invDerivative} to find
\begin{align*}
	\evalhzero{\ddh{Z_h}^{-1}} = \evalhzero{\ddh(X+h\deltaX)\inv} = \evalhzero{-(X+h\deltaX)\inv \deltaX (X+h\deltaX)\inv} = X\inv \deltaX X\inv\,.
\end{align*}
Thus, with integration by parts, \eqref{eq:localVariationalCalcMiddle} computes to
\begin{align*}
	\int\nolimits_0^1 \innerproduct{X^{-1} \dot{X}, -X^{-1}\deltaX X^{-1}\dot{X} + X^{-1}\dot{\deltaX}} \,\dt &= \int\nolimits_0^1 -\innerproduct{X^{-T}X^{-1} \dot{X}, \deltaX X^{-1}\dot{X}} + \innerproduct{X^{-T}X^{-1} \dot{X}, \dot{\deltaX}} \,\dt\nonumber\\
	&= \int\nolimits_0^1 -\innerproduct{X^{-T}X^{-1} \dot{X} (X^{-1}\dot{X})^T, \deltaX} - \innerproduct{\ddt (X^{-T}X^{-1} \dot{X}), \deltaX} \,\dt \nonumber\\
	&= -\int\nolimits_0^1 \innerproduct{X^{-T}X^{-1} \dot{X} (X^{-1}\dot{X})^T+\ddt (X^{-T}X^{-1} \dot{X}), \,\deltaX} \,\dt\,.
\end{align*}
Since this equation holds for all $\deltaX\in C^1_0([0,1];\gln)$, we can apply the fundamental lemma of the calculus of variations to obtain the differential equation
\begin{equation}
	\ddt (X^{-T}X^{-1}\dot{X}) = -X^{-T}X^{-1}\dot{X}(X^{-1}\dot{X})^T\,.\label{eq:simpleVariationBasicODE}
\end{equation}
We compute the left hand term (using the product rule for matrix valued functions \cite{NeffCISMnotes} as well as, again, Lemma \ref{lemma:invDerivative}):
\begin{align}
	\ddt (X^{-T}X^{-1}\dot{X}) &= (-X^{-T}\dot{X}^TX^{-T})(X^{-1}\dot{X}) + X^{-T}\ddt(X^{-1}\dot{X})\nonumber\\
	&= -X^{-T}\dot{X}^TX^{-T}X^{-1}\dot{X} + X^{-T}(-X^{-1}\dot{X}X^{-1}\dot{X} + X^{-1}\ddot{X})\nonumber\\
	&= X^{-T}(-\dot{X}^TX^{-T}X^{-1}\dot{X} - X^{-1}\dot{X}X^{-1}\dot{X} + X^{-1}\ddot{X})\,. \label{eq:simpleVariationExpandedDerivative}
\end{align}
To simplify notation we define $\tangent \colonequals X^{-1}\dot{X}$. Then
\begin{equation}
	\dot{\tangent} = -X^{-1}\dot{X}X^{-1}\dot{X} + X^{-1}\ddot{X} = -\tangent\tangent + X^{-1}\ddot{X}\,,\label{eq:tangentDerivative}
\end{equation}
and combining \eqref{eq:simpleVariationBasicODE} with \eqref{eq:simpleVariationExpandedDerivative} we obtain
\begin{alignat}{2}
& & -X^{-T}\overbrace{X^{-1}\dot{X}}^{=\tangent} \overbrace{(X^{-1}\dot{X})^T}^{=\tangent^T} &= X^{-T}(-\overbrace{(X^{-1}\dot{X})^TX^{-1}\dot{X}}^{=\tangent^T\tangent} - \overbrace{X^{-1}\dot{X}X^{-1}\dot{X}}^{=\tangent\tangent} + X^{-1}\ddot{X})\nonumber\\
&\Leftrightarrow & -\tangent\tangent^T &= -\tangent^T \tangent - \tangent\tangent + X^{-1}\ddot{X}\nonumber\\
&\Leftrightarrow & \tangent^T \tangent - \tangent\tangent^T &= -\tangent\tangent + X^{-1}\ddot{X}\nonumber\\
&\overset{\mathclap{\eqref{eq:tangentDerivative}}}{\Leftrightarrow} & \dot{\tangent} &= \tangent^T \tangent - \tangent\tangent^T.\label{eq:canonicalDifferentialEquation}
\end{alignat}
Equation \eqref{eq:canonicalDifferentialEquation} can also be written as
$
	\dot\tangent = [\tangent^T\!, \tangent]\:,
$
where
$
	[A,B]=AB-BA
$
denotes the \emph{commutator} on $\gln$. Vandereycken et al.\ \cite{Vandereycken2010} give an analogous equation for the case of the canonical right-invariant metric.

\subsection{The geodesic equation for the general metric}
\label{subsection:generalDifferentialEquation}
Let us now consider the general case of arbitrary $\mu,\mu_c,\kappa>0$. In order to find the geodesic equation for this case, we need the following lemma.

\begin{lemma}
	\label{lemma:productOfCommutators}
	Let $A,B\in \gln$. Then
	\begin{equation}
		\fullproduct{A,[A,B]} = \frac{\mu+\mu_c}{2}\eucproduct{A,[A,B]}\,.
	\end{equation}
\end{lemma}
\begin{proof}
	Since $\tr(AB)=\tr(BA)$, we immediately see that
	\begin{equation}
		\tr(AB-BA) = \tr(AB)-\tr(BA)=0\,. \label{eq:traceCommutes}
	\end{equation}
	Therefore
	\begin{align}
		\eucproduct{\sym A,\sym(AB-BA)} &= \eucproduct{\dev\sym A,\dev\sym(AB-BA)}+\frac1n\,\tr A\, \tr (AB-BA)\nonumber\\
		&= \eucproduct{\dev\sym A,\dev\sym(AB-BA)}\label{eq:symProdEqualsDevSymProd}
	\end{align}
	and
	\begin{align}
		&\hspace{-7mm}\eucproduct{\sym A, \sym(AB-BA)} - \eucproduct{\skew A, \skew(AB-BA)}\nonumber \\
		&= \eucproduct{\sym A, \,AB-BA} - \eucproduct{\skew A, \,AB-BA}\nonumber\\
		&= \eucproduct{\sym A - \skew A,\, AB-BA} \;=\; \eucproduct{A^T, \,AB-BA} = \tr(AAB)-\tr(ABA) = 0\,,\label{eq:symSkewEqual}
	\end{align}
	thus
	\begin{align}
		\fullproduct{A,[A,B]} \;&=\; \fullproduct{A,(AB-BA)}\nonumber\\
		&\overset{\mathclap{\eqref{eq:traceCommutes}}}{=}\; \mu \eucproduct{\dev \sym A, \dev \sym (AB-BA)}+ \mu_c \eucproduct{\skew A, \skew (AB-BA)}\nonumber\\
		&\overset{\mathclap{\eqref{eq:symProdEqualsDevSymProd}}}{=}\; \mu \eucproduct{\sym A, \sym (AB-BA)}+ \mu_c \eucproduct{\skew A, \skew (AB-BA)}\,.\nonumber\\
		&\overset{\mathclap{\eqref{eq:symSkewEqual}}}{=}\; (\mu+\mu_c)(\eucproduct{\sym A, \sym (AB-BA)}\nonumber\\
		&\overset{\mathclap{\eqref{eq:symSkewEqual}}}{=}\; \frac{(\mu+\mu_c)}{2}(\eucproduct{\sym A, \sym (AB-BA)} + \eucproduct{\skew A, \skew (AB-BA)})\nonumber\\
		&=\; \frac{(\mu+\mu_c)}{2}\eucproduct{A, (AB-BA)}\,,\nonumber\qedhere\\\nonumber
	\end{align}
	which completes the proof.
\end{proof}
Let $X\in\adsetAB$ be a \emph{piecewise two-times differentiable} energy minimizer with regard to the Riemannian metric
\begin{align}
	g_Z(X,Y) &\colonequals \fullproduct{Z^{-1}X, Z^{-1}Y}\,,\nonumber\\
	\fullproduct{A,B} &\colonequals \mu \eucproduct{\dev \sym A, \dev \sym B}  + \mu_c \eucproduct{\skew A, \skew B} + \frac{\kappa}{n} \tr A \tr B \label{eq:generalProductDefinition}\,.
\end{align}
Note that, in contrast to Section \ref{subsection:canonicalDifferentialEquation}, we do not require $X$ to be differentiable everywhere, but $\eval{X}{[a_k,a_{k+1}]}\in C^2([a_k,a_{k+1}];\GLp(n))$ for $k\in\{1,\dotsc,m\}$. Furthermore, following an approach by Lee et al.\ \cite{LeeCKP2007}, we consider a different variation: for $Y_h \colonequals X(\id + h\deltaX)$, $\deltaX\in C^1_0([0,1];\gln)$, the equation
\begin{equation}
	0 = \frac{1}{2}\evalhzero{\ddh \ener(Y_h)}\label{eq:alternateVariation}
\end{equation}
holds for the energy minimizer $X$. Again, define $\tangent \colonequals X^{-1}\dot{X}$, as well as $\tangent_h \colonequals {Y_h}^{-1}\dot{Y}_h$.
Then
\begin{align}
	&\hspace{-7mm}\,Y_h(\tangent+h\overbrace{(\tangent \deltaX-\deltaX\tangent+\dot{\deltaX})}^{\equalscolon C}) - h^2XV\overbrace{(\tangent \deltaX-\deltaX\tangent+\dot{\deltaX})}^{=C})\nonumber\\
	&= Y_h(\tangent + hC) - h^2XVC \nnl
	&= (X+hX\deltaX)(\tangent+hC) - h^2XVC\nonumber\\
	&= X\tangent + hX\deltaX\tangent + hXC + h^2 X\deltaX C - h^2XVC\nonumber\\
	&= \dot{X} + hX\deltaX\tangent + hX\tangent \deltaX - hX\deltaX\tangent + hX\dot{\deltaX}\nonumber\\
	&= \dot{X} + hX\tangent\deltaX + hX\dot{\deltaX} \;=\; \dot{X} + h\dot{X}\deltaX + hX\dot{\deltaX} \;=\; \dot{Y}_h\label{eq:localFullVariationYhdot}\,.
\end{align}
We therefore obtain
\begin{align}
	\tangent_h = Y_h\inv\dot{Y}_h \;\;&\overset{\mathclap{\eqref{eq:localFullVariationYhdot}}}{=}\;\; Y_h\inv(Y_h(\tangent + hC) - h^2XVC)\\
	&= \tangent + hC - h^2 \,Y_h\inv X\deltaX C)\nonumber\\
	&= \tangent + h(\tangent \deltaX-\deltaX\tangent+\dot{\deltaX}) - h^2 \,Y_h\inv X\deltaX C\,,
\end{align}
which implies
\begin{align}
	\evalhzero{\tangent_h} = \tangent, \qquad \ddh\evalhzero{\tangent_h} = \tangent \deltaX-\deltaX\tangent+\dot{\deltaX}\,. \label{eq:localFullVariationTangenth}
\end{align}
Coming back to equation \eqref{eq:alternateVariation}, we find that
\begin{align}
	0 = \frac{1}{2}\evalhzero{\ddh \ener(Y_h)} &= \frac{1}{2}\int\nolimits_0^1 \evalhzero{\ddh \fullproduct{{Y_h}^{-1}\dot{Y}_h , {Y_h}^{-1}\dot{Y}_h}} \dt\nonumber\\
	&= \frac{1}{2}\int\nolimits_0^1 \evalhzero{\ddh \fullproduct{\tangent_h,\tangent_h}} \dt\nonumber\\
	&= \int\nolimits_0^1 \evalhzero{\fullproduct{\tangent_h,\ddh \tangent_h}} \dt \overset{\eqref{eq:localFullVariationTangenth}}{=} \int\nolimits_0^1 \fullproduct{\tangent,\tangent \deltaX-\deltaX\tangent+\dot{\deltaX}} \,\dt\,.\label{eq:alternateVariationIntegralForm}
\end{align}
Using Lemma \ref{lemma:productOfCommutators}, we can write \eqref{eq:alternateVariationIntegralForm} as
\begin{align}
	&\hspace{-7mm}\int\nolimits_0^1 \fullproduct{\tangent,\dot{\deltaX}} + \fullproduct{\tangent,\tangent \deltaX-\deltaX\tangent} \,\dt\nonumber\\
	&= \int\nolimits_0^1 \fullproduct{\tangent,\dot{\deltaX}} + \frac{\mu+\mu_c}{2}\eucproduct{\tangent,\tangent \deltaX-\deltaX\tangent} \,\dt\nonumber\\
	&= \int\nolimits_0^1 \fullproduct{\tangent,\dot{\deltaX}} + \frac{\mu+\mu_c}{2}\eucproduct{\deltaX,\smash{\underbrace{\tangent^T\tangent-\tangent\tangent^T}_{\in \Symn}}} \,\dt\nonumber\\
	&= \int\nolimits_0^1 \fullproduct{\tangent,\dot{\deltaX}} + \frac{\mu+\mu_c}{2\mu}(\mu\eucproduct{\dev\sym \deltaX,\dev\sym (\tangent^T\tangent-\tangent\tangent^T)}\nonumber\\ &\qquad\qquad+\underbrace{\mu_c\eucproduct{\skew \deltaX,\skew (\tangent^T\tangent-\tangent\tangent^T)} + \frac{\kappa}{n}\tr \deltaX \tr(\tangent^T\tangent-\tangent\tangent^T))}_{=0} \,\dt\nonumber\\
	&= \int\nolimits_0^1 \fullproduct{\tangent,\dot{\deltaX}} + \frac{\mu+\mu_c}{2\mu}\fullproduct{\deltaX,\smash{\tangent^T\tangent-\tangent\tangent^T}} \,\dt\nonumber\\
	&= \sum_{k=0}^m \int\nolimits_{a_k}^{a_{k+1}} \fullproduct{\tangent,\dot{\deltaX}} + \frac{\mu+\mu_c}{2\mu}\fullproduct{\deltaX,\smash{\tangent^T\tangent-\tangent\tangent^T}} \,\dt\,,
\end{align}
where $a = a_0 < a_1 < \dots < a_{m+1} = b$ are chosen such that $\eval{X}{[a_j,a_{j+1}]} \in C^2([a_j,a_{j+1}];\GLn)$ for all $j\in\{0,\dotsc,m\}$.
For $\deltaX\in C^1_0([0,1];\gln)$ with $\deltaX(a_k)=0$ for $k\in\{1,\dotsc,m\}$, integration by parts yields
\begin{align}
	0 &= \sum_{k=0}^m \int\nolimits_{a_k}^{a_{k+1}} \fullproduct{\tangent,\dot{\deltaX}} + \frac{\mu+\mu_c}{2\mu}\fullproduct{\deltaX,\smash{\tangent^T\tangent-\tangent\tangent^T}} \,\dt\nonumber\\
	&= \sum_{k=0}^m \int\nolimits_{a_k}^{a_{k+1}} -\fullproduct{\smash{\deltaX},\dot\tangent} + \frac{\mu+\mu_c}{2\mu}\fullproduct{\deltaX,\smash{\tangent^T\tangent-\tangent\tangent^T}} \,\dt\,.
\end{align}
We can now apply the fundamental lemma of the calculus of variations to find
\begin{align}
	\dot{\tangent}(t) = \frac{1+\frac{\mu_c}{\mu}}{2}(\tangent(t)^T\tangent(t)-\tangent(t)\tangent(t)^T) \quad\: \forall \,t\in[0,1]\backslash\{a_1,\dotsc,a_m\}\,. \label{eq:generalDifferentialEquationNoRegularity}
\end{align}

To show that $X$ is two-times differentiable on $[0,1]$ and that equation \eqref{eq:generalDifferentialEquationNoRegularity} is satisfied everywhere, we first show that $\tangent = X^{-1}\dot X$ is continuous at each $a_k$ (and therefore on the whole interval $[0,1]$). We choose $\deltaX\in C^1_0((a_{k-1},a_{k+1});\gln)$ for $k\in\{1,\dotsc,m\}$ and compute
\begin{align}
	0 &= \int\nolimits_0^1 \fullproduct{\tangent,\dot{\deltaX}} + \frac{\mu+\mu_c}{2\mu}\fullproduct{\deltaX,\smash{\tangent^T\tangent-\tangent\tangent^T}} \,\dt\nonumber\\
	&= \int\nolimits_{a_{k-1}}^{a_k} \fullproduct{\tangent,\dot{\deltaX}} + \frac{\mu+\mu_c}{2\mu}\fullproduct{\deltaX,\smash{\tangent^T\tangent-\tangent\tangent^T}} \,\dt\nonumber\\
	&\quad\: + \int\nolimits_{a_{k}}^{a_{k+1}} \fullproduct{\tangent,\dot{\deltaX}} + \frac{\mu+\mu_c}{2\mu}\fullproduct{\deltaX,\smash{\tangent^T\tangent-\tangent\tangent^T}} \,\dt\nonumber\\
	&= \int\nolimits_{a_{k-1}}^{a_k} -\fullproduct{\deltaX,\smash{\dot\tangent}} + \frac{\mu+\mu_c}{2\mu}\fullproduct{\deltaX,\smash{\tangent^T\tangent-\tangent\tangent^T}} \,\dt\: + \left[\fullproduct{\deltaX(t),\smash{\tangent}(t)}\right]_{t=a_{k-1}}^{a_k}\nonumber\\
	&\quad + \int\nolimits_{a_k}^{a_{k+1}} -\fullproduct{\deltaX,\smash{\dot\tangent}} + \frac{\mu+\mu_c}{2\mu}\fullproduct{\deltaX,\smash{\tangent^T\tangent-\tangent\tangent^T}} \,\dt\: + \left[\fullproduct{\deltaX(t),\smash{\tangent}(t)}\right]_{t=a_k}^{a_{k+1}}\,. \label{eq:variationalRegularityUnfinishedIntegral}
\end{align}
According to \eqref{eq:generalDifferentialEquationNoRegularity} the equality
\begin{equation}
	\dot{\tangent}(t) = \frac{\mu+\mu_c}{2\mu}(\tangent(t)^T\tangent(t)-\tangent(t)\tangent(t)^T) \nonumber \\
\end{equation}
holds on $(a_{k-1},a_k)$ as well as $(a_k,a_{k+1})$. Thus the integrals in \eqref{eq:variationalRegularityUnfinishedIntegral} are zero and we find
\begin{align}
	0 &= \quad\left[\fullproduct{\deltaX(t),\smash{\tangent}(t)}\right]_{t=a_{k-1}}^{a_k} + \left[\fullproduct{\deltaX(t),\smash{\tangent}(t)}\right]_{t=a_k}^{a_{k+1}}\nonumber\\
	&= \lim_{t\nearrow a_k} \fullproduct{\deltaX(t),\smash{\tangent}(t)} \:-\: \lim_{t\searrow a_k} \fullproduct{\deltaX(t),\smash{\tangent}(t)}\:.
\end{align}
Since $\deltaX$ can be chosen with arbitrary values for $\deltaX(a_k)$, we find
\begin{align}
	\lim_{t\nearrow a_k} \tangent(t) \:&=\: \lim_{t\searrow a_k} \tangent(t)\:.
\end{align}
Therefore $\tangent$ is continuous on $[0,1]$. But then \eqref{eq:generalDifferentialEquationNoRegularity} implies that $\dot\tangent$ is continuous as well. Thus $\Xdot = X \tangent$ is continuous on the whole interval, and therefore $X$ is continuously differentiable. But then $\Xdot = X \tangent$ is continuously differentiable as well, and thus $X\in C^2([0,1];\GLp(n))$, and $\tangent=X^{-1}\dot X$ satisfies the \emph{geodesic equation}
\begin{align}
	\dot{\tangent} = \frac{1+\frac{\mu_c}{\mu}}{2}(\tangent^T\tangent-\tangent\tangent^T)\label{eq:generalDifferentialEquation}
\end{align}
everywhere on $[0,1]$.
The results of Section \ref{subsection:generalDifferentialEquation} can be summarized as follows:
\begin{proposition}
\label{proposition:odeForPiecewiseSmoothEnergyMinimizers}
Let $A,B \in \GLp(n)$ and let $X\in\adsetAB$ be a piecewise two-times differentiable energy minimizer with regard to the left-invariant Riemannian metric induced by the inner product $\fullproduct{\cdot,\cdot}$. Then $X\in C^2([a,b];\GLp(n))$, and $\tangent=X\inv\Xdot$ is a solution to \eqref{eq:generalDifferentialEquation}.
\end{proposition}
\begin{remark}
\label{remark:lengthMinimizersAreGeodesics}
	In particular, every length minimizing curve can be reparameterized to an energy minimizer, according to Lemmas \ref{lemma:constantSpeed} and \ref{lemma:energyLengthMinimizerRelation}. Thus every \emph{length minimizer} has a reparameterization that satisfies the geodesic equation.
\end{remark}
\begin{remark}
	Equation \eqref{eq:generalDifferentialEquation} could also be deduced from a more general formula given by Bloch et al.\ \cite{bloch2011}, which is derived via the Euler-Poincar\'e equations corresponding to the length minimization problem \cite[p. 430ff.]{marsden1999}.
\end{remark}

\subsection{Properties of solutions $\tangent$}
Some properties of the solutions $\tangent\in C^1([0,1];\gln)$ to \eqref{eq:generalDifferentialEquation} can be inferred directly from the equation.
\begin{lemma}
\label{lemma:geodesicProperties}
	Let $\tangent$ be a solution to \eqref{eq:generalDifferentialEquation}. Then:
	\begin{alignat}{2}
		&i) &\fullnorm{\smash{\tangent}} &\equiv \mathrm{constant}\,,\label{eq:odeNormEstimate}\\
		&ii) & \tr (\tangent) &\equiv \mathrm{constant}\,,\label{eq:odeTraceEstimate}\\
		&iii)\: & \det (\tangent) &\equiv \mathrm{constant}\,,\label{eq:odeDetEstimate}\\
		&iv)\quad & \tr (\Cof \tangent) &\equiv \mathrm{constant}\,,\label{eq:odeTrCofEstimate}
	\end{alignat}
	where $\Cof \tangent$ denotes the \emph{cofactor} of $\tangent$.
\end{lemma}
\begin{proof}
	i)
	\begin{align}
		\ddt \fullnorm{\smash\tangent}^2 &= \ddt\fullproduct{\tangent, \tangent} = 2 \fullproduct{\tangent, \smash{\dot\tangent}} = (1+\tfrac{\mu_c}{\mu}) \fullproduct{\tangent,\tangent^T\tangent - \tangent\tangent^T}\nonumber\\
		&= (1+\tfrac{\mu_c}{\mu}) \Big(\mu\fullproduct{\sym\tangent,\tangent^T\tangent - \tangent\tangent^T} + \mu_c\fullproduct{\skew\tangent,\smash{\underbrace{\tangent^T\tangent - \tangent\tangent^T}_{\in\Sym(n)}}}\nonumber\\
		&\qquad+\frac{\kappa}{n}\tr\tangent\underbrace{\tr(\tangent^T\tangent - \tangent\tangent^T)}_{=0}\Big)\nonumber\\
		&= (\mu+\mu_c)\eucproduct{\half(\tangent+\tangent^T),\tangent^T\tangent - \tangent\tangent^T}\nonumber\\
		&= \frac{\mu+\mu_c}{2}\tr(\tangent^T\tangent\tangent - \tangent\tangent^T\tangent + \tangent^T\tangent^T\tangent - \tangent^T\tangent\tangent^T)\nonumber\\
		&= \frac{\mu+\mu_c}{2}\tr(\tangent^T\tangent\tangent) - \tr(\tangent\tangent^T\tangent) + \tr(\tangent^T\tangent^T\tangent) - \tr(\tangent^T\tangent\tangent^T) \;=\; 0\,.
	\end{align}
	ii)
	\begin{align}
		\ddt \tr\tangent &= \tr\dot\tangent = \frac{1+\frac{\mu_c}{\mu}}{2}\tr(\tangent^T\tangent - \tangent\tangent^T) = 0\,.\nnl
	\end{align}
	iii)
	\begin{align}
		\ddt \det\tangent &= \tr(\Adj(\tangent)\dot\tangent)\label{eq:localUsedDetDerivative} = \tfrac{1+\frac{\mu_c}{\mu}}{2}\tr(\Adj(\tangent)(\tangent^T\tangent - \tangent\tangent^T))\nonumber\\
		&= \tfrac{1+\frac{\mu_c}{\mu}}{2}\tr(\Adj(\tangent)\tangent^T\tangent - \Adj(\tangent)\tangent\tangent^T)\nonumber\\
		&= \tfrac{1+\frac{\mu_c}{\mu}}{2}\tr(\tangent^T\tangent\Adj(\tangent) - \det(\tangent)\tangent^T) = \tfrac{1+\frac{\mu_c}{\mu}}{2} \tr(\det(\tangent)\tangent^T - \det(\tangent)\tangent^T) \;=\; 0\,,
	\end{align}
	where $\Adj(A) = (\Cof A)^T$ is the adjugate matrix of $A$ with $\Adj(A)\cdot A = A\cdot \Adj(A) = (\det A) \, \id$. For the derivative of the determinant function, see e.g.\ \cite{NeffCISMnotes}.\\[.5em]
	iv) We first assume that $\tangent$ is invertible. Then
	\begin{align}
		\ddt \tr\tangent\inv &= -\tr(\tangent\inv \dot\tangent \tangent\inv) \nnl
		&= -\tfrac{1+\frac{\mu_c}{\mu}}{2}\tr(\tangent\inv (\tangent^T\tangent - \tangent\tangent^T) \tangent\inv) = -\tfrac{1+\frac{\mu_c}{\mu}}{2}\tr(\tangent\inv \tangent^T - \tangent^T\tangent\inv) = 0\,,
	\end{align}
	and since $\det(\tangent)$ is constant as well, we find
	\[
		\tr\Cof(\tangent) = \det(\tangent)\cdot \tr \tangent^{-T} = \det(\tangent)\cdot \tr (\tangent^{-1}) \equiv \text{constant}\,.
	\]
	The general case follows directly from the density of $\GLn$ in $\gln=\Rnn$.
\end{proof}

\subsection{Existence and uniqueness of solutions}
Since the mapping $\tangent\mapsto \frac{1+\frac{\mu_c}{\mu}}{2} (\tangent^T\tangent - \tangent\tangent^T)$ is locally Lipschitz continuous, the Picard–Lindelöf theorem implies that equation \eqref{eq:generalDifferentialEquation} has a unique local solution for given $\tangent(0)$ and that if a global solution exists, it is uniquely defined. Similarly, for given $\tangent\in C^1(I,\gln)$ on an arbitrary interval $I$, the initial value problem
\begin{align}
	\left\{\begin{aligned}
		\dot X &= X\tangent\,,\\
		X(0) &= X_0
	\end{aligned}\right.
\end{align}
has a unique local solution as well. Using estimates \eqref{eq:odeNormEstimate} to \eqref{eq:odeDetEstimate}, we can even show that this solution is globally defined:
\begin{proposition}
	\label{prop:generalIVP}
	For $X_0\in\GLp(n)$, $M_0\in\gln$, the initial value problem
	\begin{align}
		\left\{\begin{aligned}
			\tangent &= X^{-1}\dot X\,,\\
			\dot \tangent &= \frac{1+\frac{\mu_c}{\mu}}{2}(\tangent^T\tangent - \tangent\tangent^T)\,,\\
			X&(0) = X_0\,,\quad \dot X(0) = M_0
		\end{aligned}\right.\label{eq:generalIVP}
	\end{align}
	for the geodesic equation has a unique global solution $X\colon\R\to\GLp(n)$.
\end{proposition}
\begin{proof}
	Since $\tangent(0)=X(0)^{-1}\dot X(0)$ is determined by $X(0)$ and $\dot X(0)$, \eqref{eq:generalIVP} has a unique local solution, as demonstrated above. Then, according to \eqref{eq:odeNormEstimate}, $\fullnorm{\smash\tangent}$ is constant, hence $\tangent$ is bounded in $\Rnn$ and therefore defined on all of $\R$. The same estimate shows that the mapping $X\mapsto\tangent X$ is globally Lipschitz continuous. Therefore (after rewriting \eqref{eq:generalIVP}$_1$ as $\Xdot=UX$) there exists a unique global solution $X\colon\R\to\Rnn$ to \eqref{eq:generalIVP}.\\
	It remains to show that $X(t)\in\GLp(n)$ for all $t\in\R$. First we observe that, for $\det X>0$,
	\begin{align}
		\ddt(\det X)&=\tr(\Adj(X)\dot X) = \tr(\det(X)X^{-1}\dot X) = \det(X) \tr\tangent \overset{\eqref{eq:odeTraceEstimate}}{=} \det (X) \tr(\tangent(0))\,.
	\end{align}
	Thus on every interval $I\subseteq\R$ with $0\in I$ and $X(I)\subseteq\GLp(n)$, the function $t\mapsto\det X(t)$ solves the initial value problem
	\begin{align}
		\left\{\begin{aligned}
			\ddt(\det X(t)) &= \tr(\tangent(0))\det X(t)\,,\\
			\det X(0) &= \det X_0\,.
		\end{aligned}\right.\label{eq:detODE}
	\end{align}
	Assume now that $X(\R)\nsubseteq\GLp(n)$. Since $I^+\colonequals\{t\in\R \setvert \det X(t)>0\}$ is open and $0\in I^+$ by definition of $X_0$, there is a minimal $t_0>0$ with $-t_0\notin I^+$ or $t_0\notin I^+$. But then $\det X$ is the (unique) solution to \eqref{eq:detODE}, i.e.
	\begin{align}
		\det X(t) &= (\det X_0)\,e^{t\cdot \tr(\tangent(0))}
	\end{align}
	on $(-t_0, t_0)$, and thus
	\[
		0<\,\lim_{\mathclap{t\searrow-t_0}} \det X(t) = \det X(-t_0)\,, \qquad 0<\,\lim_{\mathclap{t\nearrow t_0}} \det X(t) = \det X(t_0)\,,
	\]
	in contradiction to the definition of $t_0$.
\end{proof}

\section{The parameterization of geodesics}
\label{section:parameterization}
\subsection{Solving the geodesic equation}
We will now give an explicit solution to the geodesic initial value problem \eqref{eq:generalIVP}. The result is inspired by the work of Mielke \cite{Mielke2002}, who obtained formula \eqref{eq:definitionGeodesic} as a parametrization of the geodesics on $\SL(n)$. For the canonical inner product, i.e.\ for the special case $\mu=\muc=\kappa=1$, the result can also be found in \cite{Andruchow2011}. Note that while geodesics on a Lie group equipped with a bi-invariant Riemannian metric are translated one-parameter groups which can easily be characterized through the matrix exponential \cite{Taubes2011} (which, for example, allows for an easy explicit computation of the geodesic distance on the special orthogonal group $\SOn$ with respect to the canonical metric \cite{Moakher2002}), the metric discussed here is not bi-$\GLn$-invariant and thus the solution to the geodesic equation takes on a more complicated form.

Let $A\in\GLp(n)$, $M\in\gln$ and $\omega = \frac{\mu_c}{\mu}$. We define
\begin{align}
	\vPhi(M) \ratio&= \exp(\sym M - \omega \skew M)\exp((1+\omega)\skew M)\,,\label{eq:definitionPhi}\\
	\vPhiA(M) \ratio&= A\cdot\vPhi(M) = A\exp(\sym M - \omega \skew M)\exp((1+\omega)\skew M)\,,\label{eq:definitionPhiA}\\
	X(t) \ratio&= \vPhiA(tM) = A\exp(t(\sym M - \omega \skew M))\exp(t(1+\omega)\skew M)\,,\label{eq:definitionGeodesic}
\end{align}
where $\exp\colon \gln\to\GLpn$ denotes the matrix exponential (cf.\ Appendix \ref{section:expAppendix}). We use equation \eqref{eq:expSimpleDerivative} to compute
\begin{align}
	\dot X &= A\exp(t(\sym M - \omega \skew M))\:(\sym M - \omega \skew M)\:\exp(t(1+\omega)\skew M)\nonumber\\
	&\quad+ A\exp(t(\sym M - \omega \skew M))\:(1+\omega)\skew M\:\exp(t(1+\omega)\skew M)\nonumber\\
	&= A\exp(t(\sym M - \omega \skew M))\:(\sym M + \skew M)\:\exp(t(1+\omega)\skew M)\nnl
	&= A\exp(t(\sym M - \omega \skew M))M\exp(t(1+\omega)\skew M)\label{eq:derivativeGeodesic}
\end{align}
and
\begin{align}
	\tangent \colonequals X^{-1}\dot X &= \exp(t(1+\omega)\skew M)^{-1}\exp(t(\sym M - \omega \skew M))^{-1}A^{-1}\nonumber\\
	&\quad\cdot A\exp(t(\sym M - \omega \skew M))\:M\:\exp(t(1+\omega)\skew M)\nonumber\\
	&= \exp(t(1+\omega)\skew M)^{-1} \:M\: \underbrace{\exp(\overbrace{t(1+\omega)\skew M}^{\in\so(n)})}_{\equalscolon Q\in\SO(n)} \;=\; Q^T M Q\,, \label{eq:geodesicTangent}
\end{align}
where we used \eqref{eq:expsoSO} stated in the appendix to infer $Q\in\SOn$. Then
\begin{align}
	\tangent^T\tangent - \tangent\tangent^T &= Q^T M^TM Q - Q^T MM^T Q = Q^T(M^TM - MM^T)Q
\end{align}
and
\begin{align}
	\dot\tangent = \dot Q^T M Q + Q^T M \dot Q \quad&\overset{\mathclap{\eqref{eq:expSimpleDerivative}}}{=}\quad Q^T (1+\omega)(\skew M)^T M Q + Q^T (1+\omega)M(\skew M) Q\nonumber\\
	&= (1+\omega)Q^T( (\skew M)^T M + M (\skew M) )Q\nonumber\\
	&= (1+\omega)Q^T( M (\skew M) - (\skew M) M)Q\nonumber\\
	&= (1+\omega)Q^T(\frac{1}{2}(MM - MM^T) - \frac{1}{2}(MM - M^TM))Q\nonumber\\
	&= \frac{(1+\omega)}{2}Q^T(M^TM - MM^T)Q \;=\; \frac{(1+\omega)}{2}(\tangent^T\tangent - \tangent\tangent^T)\,,
\end{align}
and thus $X$ solves the geodesic equation. To solve the initial value problem \eqref{eq:generalIVP} we use the equality $\exp(0)=\id$ as well as equation \eqref{eq:derivativeGeodesic} to obtain
\begin{align}
	X(0)=A\,, \quad \dot X(0) &= AM\,,\label{eq:derivativeGeodesicZero}
\end{align}
hence we can solve \eqref{eq:generalIVP} by choosing $A=X_0$, $M=X_0\inv M_0$. We conclude:
\begin{theorem}
\label{theorem:parameterizationOfGeodesics}
	Let $X_0\in\GLpn,\, M_0\in\gln$. Then the curve $X\colon\R \to \GLpn$ with
	\begin{align}
		X(t) &= \vPhi_{X_0}(t\,X_0^{-1}M_0) = X_0\exp(t(\sym (X_0^{-1}M_0) - \tfrac{\mu_c}{\mu} (X_0^{-1}M_0)))\exp(t(1+\tfrac{\mu_c}{\mu})\skew (X_0^{-1}M_0))\nonumber
	\end{align}
	is the unique global solution to the geodesic initial value problem
	\[\pushQED{\qed}
		\left\{\begin{aligned}
			\tangent &= X^{-1}\dot X\,,\\
			\dot \tangent &= \frac{1+\frac{\mu_c}{\mu}}{2}(\tangent^T\tangent - \tangent\tangent^T)\,,\\
			X(0) &= X_0\,,\quad \dot X(0) = M_0\,.\mathrlap{\hspace{49mm}\qedhere}
		\end{aligned}\right.
	\]\popQED
\end{theorem}
\noindent Because of the left-invariance of $g$, we can mostly focus on the case $A=\id$, $\vPhiA = \vPhi_\id = \vPhi$ and
\[
	X(t) = \exp(t(\sym M - \omega \skew M))\exp(t(1+\omega)\skew M)\,.
\]

\subsection{Properties of geodesic curves}
Since the geodesic curves
\begin{equation}
	X(t) =  A\exp(t(\sym M - \omega \skew M))\exp(t(1+\omega)\skew M)
\end{equation}
solve the geodesic initial value problem \eqref{eq:generalIVP}, the properties given in Lemma \ref{lemma:geodesicProperties} can be directly applied to $\tangent = X\inv \Xdot$. Furthermore, we can compute the length of $X$ on the interval $[0,t_0]$ for given $M\in\gln$ (recall from \eqref{eq:geodesicTangent} that $X\inv\Xdot = \tangent = Q^TMQ$ with $Q(t)\in\SOn$):
\begin{align}
	\len(X) = \int_0^{t_0}g_X(\dot X, \dot X)\,\dt &= \int_0^{t_0} \sqrt{\fullproduct{X^{-1}\dot X, X^{-1}\dot X}}\,\dt\nonumber\\
	&= \int_0^{t_0} \sqrt{\fullproduct{\tangent, \tangent}}\,\dt = \int_0^{t_0} \sqrt{\fullproduct{Q^T M Q, Q^T M Q}}\,\dt\nonumber\\
	&= \int_0^{t_0} \sqrt{\fullproduct{M, M}}\dt \;=\; t_0 \fullnorm{M}\,,\label{eq:geodesicLength}
\end{align}
where the second to last equality follows from the isotropy property \eqref{eq:isotropicProperty} of the inner product $\fullproduct{\cdot,\cdot}\,$.

\section{The existence of length minimizers}
\label{section:existenceOfMinimizers}

As we have seen in Section \ref{section:calculusOfVariations} (Remark \ref{remark:lengthMinimizersAreGeodesics}), every sufficiently smooth length minimizing curve can be reparameterized to solve the geodesic equation. Thus Theorem \ref{theorem:parameterizationOfGeodesics} shows that every length minimizer, after possible reparameterization, is of the form \eqref{eq:definitionGeodesic}.

At this point, it is not clear that, for given $A,B\in\GLp(n)$, such a minimizer connecting $A$ and $B$ actually exists. In the broader setting of differential geometry it can be shown that the \emph{local} existence of minimizers is guaranteed on any (differentiable) Riemannian manifold \cite[Corollary 1.4.2]{Jost1998}. Furthermore, the global existence of the geodesic curves demonstrated in Proposition \ref{prop:generalIVP} implies that $\GLn$ is \emph{geodesically complete} with respect to $g$, and thus the \emph{Hopf-Rinow theorem} also guarantees the \emph{global} existence of length minimizers. Nonetheless, in our effort to keep this paper self contained, we will give a full proof for the existence of minimizers using only results from basic real analysis. Towards this aim, we will first show that a specific variant of \emph{Gauss's lemma} \cite[2.93]{gallot1990riemannian} holds for $g$ on $\GLn$ and continue by following the proof of the Hopf-Rinow theorem as given in \cite[Theorem 1.7.1]{Jost1998} (a similar proof can be found in \cite[Theorem 2.8]{docarmo1992}) as applied to our special case. Readers not interested in these rather basic proofs should skip the main part of Section \ref{section:existenceOfMinimizers} and continue with Theorem \ref{theorem:existenceAndParameterizationOfMinimizers}.

\subsection{Preparations}
\label{subsection:existencePreparations}
In the context of Riemannian geometry the function $\vPhiA\colon \gln\to\GLpn$ with
\begin{equation}
	\vPhiA(M) = A\exp(\sym M - \omega \skew M)\exp((1+\omega)\skew M)
\end{equation}
can be considered the \emph{geodesic exponential} at $A$ \cite[Definition 1.4.3]{Jost1998}. Note again that from this point of view the following lemma is a direct application of Gauss's lemma. We will prove it via direct computation.

\begin{lemma}
	Let $A\in\GLp(n)$, $\vM,\vT\in\gln$. Then
	\begin{equation}
		\fullg{\vPhiA(\vM)}(\vDPhiA{\vM}{\vM}, \vDPhiA{\vM}{\vT}) = \fullproduct{\vM,\vT}\,.\label{eq:gaussLemma}
	\end{equation}
\end{lemma}
\begin{remark}
	Note carefully that $M$ occurs multiple times in \eqref{eq:gaussLemma}, including the direction of one of the derivatives. This is a necessary restriction; in general, the equality $\fullg{\vPhiA(\vM)}(\vDPhiA{\vM}{S}, \vDPhiA{\vM}{\vT}) = \fullproduct{S,\vT}$ does \emph{not} hold for arbitrary $S,\vT \in \gln$.
\end{remark}
\begin{proof}
	First note that the left-invariance of the Riemannian metric $g$ implies
	\begin{align}
		\fullg{\vPhiA(\vM)}(\vDPhiA{\vM}{\vM}, \vDPhiA{\vM}{\vT}) &= \fullg{A\vPhi(\vM)}(A\vDPhi{\vM}{\vM}, A\vDPhi{\vM}{\vT})\\
		&= \fullg{\vPhi(\vM)}(\vDPhi{\vM}{\vM}, \vDPhi{\vM}{\vT})\nonumber\,.
	\end{align}
	Using \eqref{eq:expSimpleDerivative} from the appendix, we compute
	\begin{align}
		\vDPhi{\vM}{\vM} &= (\Dexp{\sym \vM - \omega \skew \vM}{(\sym \vM - \omega \skew \vM)}) \: \exp((1+\omega)\skew \vM)\nonumber\\ &\quad+ \exp(\sym \vM - \omega \skew \vM) \:(\Dexp{(1+\omega)\skew \vM}{((1+\omega)\skew \vM)})\nonumber\\
		&= \exp(\sym \vM - \omega \skew \vM)\:(\sym \vM - \omega \skew \vM)\:\exp((1+\omega)\skew \vM)\nonumber\\ &\quad+ \exp(\sym \vM - \omega \skew \vM)\:((1+\omega)\skew \vM)\:\exp((1+\omega)\skew \vM)\nonumber\\
		&= \exp(\sym \vM - \omega \skew \vM) \:(\sym\vM + \skew\vM)\: \exp((1+\omega)\skew \vM)\nnl
		&= \exp(\sym \vM - \omega \skew \vM) \:\vM\: \underbrace{\exp((1+\omega)\skew \vM)}_{\equalscolon\vQ\in\SO(n)}\nnl
		&= \exp(\sym \vM - \omega \skew \vM) \:\vM\: Q\,,
	\end{align}
	while \eqref{eq:expDerivative} yields
	{\footnotesize
	\begin{align}
		&\vDPhi{\vM}{\vT}\nnl
		&= (\Dexp{\sym \vM - \omega \skew \vM}{(\sym \vT - \omega \skew \vT)})\: \exp((1+\omega)\skew \vM)\nonumber\\
		&\qquad+ \exp(\sym \vM - \omega \skew \vM)\:(\Dexp{(1+\omega)\skew \vM}{(1+\omega)\skew \vT})\nonumber\\
		&= \int\nolimits_0^1\exp(s(\sym \vM - \omega \skew \vM))\:(\sym \vT - \omega \skew \vT) \cdot\exp((1-s)(\sym \vM - \omega \skew \vM))\:\exp((1+\omega)\skew \vM)\,\ds\nonumber\\
		&\qquad+ \int\nolimits_0^1 \exp(\sym \vM - \omega \skew \vM)\:\exp(s(1+\omega)\skew \vM)\:((1+\omega)\skew \vT) \cdot\exp((1-s)(1+\omega)\skew \vM)\,\ds\nonumber\\
		&= \exp(\sym \vM - \omega \skew \vM) \;\bigg(\int\nolimits_0^1 [\:\overbrace{\exp((s-1)(\sym \vM - \omega \skew \vM))}^{\smash{\equalscolon\vPs\in\GLp(n)}} \cdot(\sym \vT - \omega \skew \vT)\:\exp((1-s)(\sym \vM - \omega \skew \vM))\nonumber\\
		&\qquad+\exp(s(1+\omega)\skew \vM)\:((1+\omega)\skew \vT)\;\smash{\underbrace{\exp(-s(1+\omega)\skew \vM)}_{\equalscolon\vRs\in\SO(n)}}\vphantom{\int\nolimits_a^b}\:]\ds \bigg) \cdot\exp((1+\omega)\skew \vM)\nonumber\\
		&= \exp(\sym \vM - \omega \skew \vM) \int\nolimits_0^1 [\vPs (\sym \vT - \omega \skew \vT) \vPs^{-1} + (1+\omega)\vRs^T (\skew \vT) \vRs]\: \ds \cdot \exp((1+\omega)\skew \vM)\,.
	\end{align}
	}
	We also find
	\begin{align}
		\vPhi(\vM)^{-1} = \overbrace{\exp((1+\omega)\skew \vM)^{-1}}^{=\vQ^T}\:\exp(\sym \vM - \omega \skew \vM)^{-1}\,.
	\end{align}
	Hence we can compute
	{\footnotesize
	\begin{align}
		&\hspace{-12mm}\fullg{\vPhi(\vM)}(\vDPhi{\vM}{\vM},\:\vDPhi{\vM}{\vT})\nonumber\\
		&= \fullproduct{\vPhi(\vM)^{-1}\:\vDPhi{\vM}{\vM},\:\vPhi(\vM)^{-1}\:\vDPhi{\vM}{\vT}}\nonumber\\
		&= \fullproduct{\vQ^T\vM\vQ,\:\vQ^T\cdot\int\nolimits_0^1 [\vPs (\sym \vT - \omega \skew \vT) \vPs^{-1} + (1+\omega)\vRs^T (\skew \vT) \vRs]\: \ds \cdot\vQ}\nonumber\\
		&= \fullproduct{\vM, \:\int\nolimits_0^1 [\vPs (\sym \vT - \omega \skew \vT) \vPs^{-1} + (1+\omega)\vRs^T (\skew \vT) \vRs]\:\ds}\nonumber\\
		&= \int\nolimits_0^1 \fullproduct{\vM, \:\vPs (\sym \vT - \omega \skew \vT) \vPs^{-1} + (1+\omega)\vRs^T (\skew \vT) \vRs}\:\ds\nonumber\\
		&= \int\nolimits_0^1 \innerproduct{\mu\dev\sym\vM + \mu_c\skew\vM, \:\vPs (\sym \vT - \omega \skew \vT) \vPs^{-1}}\nonumber\\
		&\quad + \frac{\kappa}{n}(\tr M) \tr(\vPs (\sym \vT - \omega \skew \vT) \vPs^{-1}) + (1+\omega)\fullproduct{\vM, \:\smash{\underbrace{\vRs^T (\skew \vT) \vRs}_{\in\so(n)}}} \:\ds\nonumber\\
		&= \int\nolimits_0^1 \innerproduct{\vPs^T(\mu\dev\sym\vM + \mu_c\skew\vM)\vPs^{-T}, \:\sym \vT - \omega \skew \vT}\nonumber\\
		&\quad + \frac{\kappa}{n}(\tr M) \tr(\sym \vT - \omega \skew \vT) + (1+\omega)\,\mu_c\,\innerproduct{\skew\vM, \:\vRs^T (\skew \vT) \vRs}\:\ds\nonumber\\
		&= \int\nolimits_0^1 \mu\,\innerproduct{\vPs^T(\dev\sym\vM + \omega\skew\vM)\vPs^{-T}, \:\sym \vT - \omega \skew \vT}\nonumber\\
		&\quad + \frac{\kappa}{n}(\tr M) (\tr\vT) + (1+\omega)\,\mu_c\,\innerproduct{\vRs\skew\vM\vRs^T, \:\skew \vT}\:\ds\label{eq:gaussLemmaTerm}\,.
	\end{align}
	}
	It is not difficult to see that the matrices
	\begin{equation}
		\dev\sym\vM + \omega\skew\vM,\quad \sym \vM +\omega \skew \vM
	\end{equation}
	commute.
	Therefore, since $\vPs^T = \exp((s-1)(\sym \vM + \omega \skew \vM))$, we can use equation \eqref{eq:expCommutation} to infer that $\dev\sym\vM + \omega\skew\vM$ and $\vPs^T$ commute. Analogously, we see that $\skew M$ and $R_s=\exp(\exp((1+\omega)\skew \vM))$ commute as well. Thus expression \eqref{eq:gaussLemmaTerm} can be simplified to
	\begin{align}
		&\quad\: \int\nolimits_0^1 \mu\,\innerproduct{(\dev\sym\vM + \omega\skew\vM)\vPs^T\vPs^{-T}, \:\sym \vT - \omega \skew \vT}\\
		&\qquad\quad + \frac{\kappa}{n}(\tr M) (\tr\vT) + (1+\omega)\mu_c\innerproduct{\skew\vM\vRs\vRs^T, \skew \vT}\,\ds\nonumber\\
		&\qquad\qquad= \int\nolimits_0^1 \mu\innerproduct{\dev\sym\vM + \omega\skew\vM, \sym \vT - \omega \skew \vT}\nonumber\\
		&\qquad\qquad\qquad\qquad + \frac{\kappa}{n}(\tr M) (\tr\vT) + (1+\omega)\mu_c\innerproduct{\skew\vM, \skew \vT}\,\ds\nonumber\\
		&\qquad\qquad= \int\nolimits_0^1 \mu\innerproduct{\dev\sym\vM, \dev\sym \vT} + \mu_c\innerproduct{\skew\vM, - \omega \skew \vT}\nonumber\\
		&\qquad\qquad\qquad\qquad + \frac{\kappa}{n}(\tr M) (\tr\vT) + (\mu_c+\omega\mu_c)\innerproduct{\skew\vM, \skew \vT}\,\ds\nonumber\\
		&\qquad\qquad = \int\nolimits_0^1 \mu\innerproduct{\dev\sym\vM, \dev\sym \vT} + \mu_c\innerproduct{\skew\vM, \skew \vT} + \frac{\kappa}{n}(\tr M) (\tr\vT)\:\ds \;=\; \fullproduct{M,T}\:.\qedhere
	\end{align}
\end{proof}

This lemma can now be used to establish a lower bound for the length of curves \enquote{close to $A$}. To do so, we choose $\eps>0$ such that $\vPhiA$ is a diffeomorphism from $B_{\eps}(0)\subseteq\gln$ to an open neighbourhood of $A$ in $\GLp(n)$; note that, according to \eqref{eq:derivativeGeodesicZero}, $D\vPhiA[0]$ is surjective and therefore nonsingular. Then any \enquote{short enough} curve $Y$ with $Y(0)=A$ can be represented as $Y=\vPhiA\circ\gamma$ with a curve $\gamma$ in $B_{\eps}(0)\subseteq\gln$. To compute the length of $Y$, we need the following lemma, which is a modified version of Proposition 5.3.2 in \cite{Eschenburg2007}:

\begin{lemma}
\label{lemma:curveLengthLocalLowerEstimate}
	Let $\gamma\in C^0([0,1];B_\eps(0))$ be a piecewise differentiable curve in $B_\eps(0)\subseteq\gln$ with $\gamma(0)=0$, $\gamma(1)=\vM$. Then
	\begin{align}
		\len(\vPhiA\circ\gamma) &\ge \fullnorm{\vM}\,,\label{eq:minLengthOfShortCurves}\\
		\ener(\vPhiA\circ\gamma) &\ge \fullnorm{\vM}^2\,.
	\end{align}
\end{lemma}
\begin{proof}
	Let such a curve $\gamma$ be given. Then the length of $\vPhiA \circ \gamma$ is
	\begin{align}
		\len(\vPhiA\circ\gamma) &= \int\nolimits_0^1 \sqrt{\fullg{\vPhiA(\gamma)}\Big(\ddt(\vPhiA\circ\gamma)(t),\: \ddt(\vPhiA\circ\gamma)(t)\Big)\:} \dt\nonumber\\
		&= \int\nolimits_0^1 \sqrt{\fullg{\vPhiA(\gamma)}(\vDPhiA{\gamma}{\dot\gamma},\: \vDPhiA{\gamma}{\dot\gamma})\:}\dt\,.
	\end{align}
	Note that since $\gamma$ and $\dot\gamma$ are generally not equal, the previous lemma can not be applied directly. Therefore we decompose $\dot\gamma$ into the sum of $\gammaone$ and $\gammatwo$, where $\gammaone$ is a multiple of $\gamma$ and $\gammatwo$ is orthogonal to $\gamma$ with regard to the inner product on $\gln$: define $\vr,\gammaone,\gammatwo$ by
	\begin{align}
		\vr(t) &\colonequals	\begin{cases}
								\frac{\fullproduct{\gamma, \dot\gamma}}{\fullnorm{\gamma(t)}^2}\quad &: \quad \gamma(t)\neq 0\\
								\qquad 0 &: \quad \gamma(t) = 0
							\end{cases}\:,\\
		\gammaone(t) &\colonequals \vr(t)\,\gamma(t)\,,\nnl
		\gammatwo(t) &\colonequals \dot\gamma(t) - \gammaone(t)\,.\nonumber
	\end{align}
	Then
	\begin{align}
		\dot\gamma = \gammaone + \gammatwo\,,\qquad \fullproduct{\gamma, \gammatwo} = 0
	\end{align}
	and thus
	\begin{align}
		\len(\vPhiA\circ\gamma) &= \int\nolimits_0^1 \sqrt{\fullg{\vPhiA(\gamma)}(\vDPhiA{\gamma}{\dot\gamma},\: \vDPhiA{\gamma}{\dot\gamma})}\:\dt\nonumber\\
		&= \int\nolimits_0^1 \sqrt{\fullg{\vPhiA(\gamma)}(\vDPhiA{\gamma}{(\gammaone + \gammatwo)},\: \vDPhiA{\gamma}{(\gammaone + \gammatwo)})}\:\dt\nonumber\\
		&= \int\nolimits_0^1 \sqrt{\fullg{\vPhiA(\gamma)}(\vDPhiA{\gamma}{\gammaone} + \vDPhiA{\gamma}{\gammatwo},\: \vDPhiA{\gamma}{\gammaone} + \vDPhiA{\gamma}{\gammatwo})}\:\dt\nonumber\\
		&\ge \int\nolimits_0^1 \sqrt{\fullg{\vPhiA(\gamma)}(\vDPhiA{\gamma}{\gammaone},\: \vDPhiA{\gamma}{\gammaone}) + 2\,\fullg{\vPhiA(\gamma)}(\vDPhiA{\gamma}{\gammaone},\: \vDPhiA{\gamma}{\gammatwo})}\:\dt \label{eq:localGeodesicUniquenessInequality}\\
		&= \int\nolimits_0^1 \sqrt{\vr(t)^2\,\fullg{\vPhiA(\gamma)}(\vDPhiA{\gamma}{\gamma},\: \vDPhiA{\gamma}{\gamma}) + 2\,\vr(t)\fullg{\vPhiA(\gamma)}(\vDPhiA{\gamma}{\gamma},\: \vDPhiA{\gamma}{\gammatwo})}\:\dt\nonumber\\
		&\mathrlap{\overset{\eqref{eq:gaussLemma}}{=}}\hphantom{=\:\,} \int\nolimits_0^1 \sqrt{\vr(t)^2\,\fullproduct{\gamma, \gamma} + 2\,\vr(t)\smash{\underbrace{\fullproduct{\gamma, \gammatwo}}_{=0}} }\:\dt\,.\vphantom{\underbrace{\fullproduct{\gamma, \gammatwo}}_{=0}}\nonumber
	\end{align}
	Without loss of generality we assume $M \neq 0$. Since $\gamma$ is continuous and $\gamma(1)=M$, we can choose $a\colonequals\max\{s\in[0,1] \setvert \gamma(s)=0\}$ such that $\gamma(t) \neq 0$ for all $t>a$. We obtain
	\begin{align}
		\len(\vPhi\circ\gamma) &= \int\nolimits_0^1 \sqrt{\vr(t)^2\fullproduct{\gamma, \gamma}} \: \dt = \int\nolimits_0^1  \abs{\vr(t)}\, \fullnorm{\gamma} \:\dt \\
		&\ge \int\nolimits_a^1  \vr(t)\, \fullnorm{\gamma} \:\dt = \int\nolimits_a^1 \frac{\fullproduct{\gamma, \dot\gamma}}{\fullnorm{\gamma}} \:\dt = \int\nolimits_a^1 \ddt \fullnorm{\gamma(t)}\:\dt\,. \nonumber
	\end{align}
	The curve $\gamma$ is piecewise differentiable by assumption, so choose $a = a_1 < \dotsc <a_n < a_{n+1} = 1$ such that $\gamma$ is continuously differentiable on $(a,1) \setminus \{a_1, \dotsc, a_n\}$.\\ We compute
	\begin{align}
		\len(\vPhi\circ\gamma) &\ge \int\nolimits_a^1 \ddt \fullnorm{\gamma(t)}\:\dt = \sum_{i=1}^m\int\nolimits_{a_i}^{a_{i+1}} \ddt \fullnorm{\gamma(t)}\:\dt = \sum_{i=1}^m (\fullnorm{\gamma(a_{i+1})} - \fullnorm{\gamma(a_i)}\,)\nonumber\\
		&= \fullnorm{\gamma(a_{m+1})} - \fullnorm{\gamma(a_1)} = \fullnorm{\gamma(1)} \:-\: 0 \;=\; \fullnorm{\vM}\,.
	\end{align}
	Finally, Lemma \ref{lemma:lengthEnergyInequality} yields
	\begin{align}
		\ener(\vPhi\circ\gamma) &\ge \len(\vPhi\circ\gamma)^2 \ge \fullnorm{\vM}^2\,.\qedhere
	\end{align}
\end{proof}

We can now give a lower bound for the length of a curve $Z\in\adsetAB\setvert[0,1]\to\GLp(n)$ with $B\notin\vPhiA(B_{\eps}(0))$: with $t_0=\min\{t\in[0,1] \setvert Z(t)\notin\vPhiA(B_{\eps}(0))\}$, we can write $Z(t)=\vPhiA(\gamma(t))$ for $t\in[0,t_0)$, and $Z(t_0)\notin\vPhiA(B_{\eps}(0))$ implies
\begin{align}
	\lim_{t\nearrow t_0} \fullnorm{\gamma(t)} = \eps\,.
\end{align}
Thus, using \eqref{eq:minLengthOfShortCurves}, we find $\len(Z)\geq \eps$ and therefore
\begin{align}
	\dg(A,B) > \eps \quad \forall \, B\notin\vPhiA(B_{\eps}(0))\,.
\end{align}
Furthermore, for $\fullnorm{M}<\eps$, we can directly compute the distance between $A$ and $B=\vPhiA(M)\in B_{\eps}(A)$: For the geodesic curve $X\in\adsetAB\colon[0,1]\to\GLp(n)$ with
\begin{align}
	X(t) = \vPhiA(t\,M)
\end{align}
we find $X(0)=A,\:X(1)=\vPhiA(M)=B$ and, according to equation \eqref{eq:geodesicLength}, $\len(X)=\fullnorm{M}$. Thus $\dg(A,B)\leq\fullnorm{M}\,$. For any other curve $Y\in\adsetAB\colon[0,1]\to\GLp(n)$, we find that $Y$ is either contained in $B_\eps(A)$, in which case \eqref{eq:minLengthOfShortCurves} implies $\len(Y)\geq\fullnorm{M}\,$; or $Y$ leaves $B_\eps(A)$. But then, as shown above, $\len(Y)\geq\eps>\fullnorm{M}\,$, and therefore
\begin{align}
	\dg(A,B)=\fullnorm{M}\,,
\end{align}
and $X$ is a length minimizer. Finally these positive lower bounds imply
\begin{equation}
	\dg(A,N)>0
\end{equation}
for all $B\neq A$. Therefore $\dg$ defines a \emph{metric} on $\GLpn$. Note also that equality in \eqref{eq:localGeodesicUniquenessInequality} holds if and only if
	\begin{equation}
		\fullg{\vPhiA(\gamma)}(\vDPhiA{\gamma}{\gammatwo},\: \vDPhiA{\gamma}{\gammatwo}) = 0
	\end{equation}
	on $[0,1]$. Since $\fullg{\vPhiA(\gamma)}$ is positive definite, this is equivalent to $0=\vDPhiA{\gamma}{\gammatwo} = A\vDPhi{\gamma}{\gammatwo}$, and since $D\vPhi$ is non-singular in a neighbourhood of $0$, this equality holds (for small enough $\eps$) if and only if
	\[
		0 = \gammatwo(t) = \dot\gamma(t) - \gammaone(t) = \dot\gamma(t) - \vr(t)\,\gamma(t)
	\]
	for all $t\in [0,1]$. But then $\gamma$ and $\dot\gamma$ are linearly dependent, and thus $\gamma$ is a parameterization of a straight line in $\gln$ connecting $0$ and $M$. Hence $\vPhiA\circ\gamma$ is a reparametrization of the curve $t\mapsto\vPhiA(tM) = X(t)$, which shows that this length minimizing curve is (locally) uniquely determined.
The above considerations are summarized in the following lemma:

\begin{lemma}
	\label{lemma:geodesicSpheres}
	Let $A\in\GLp(n)$. Then there exists $\eps(A)>0$ such that for all $0<\eps<\eps(A)$:
	\begin{itemize}
	\item[i)] $\vPhiA\colon \gln\supseteq B_\eps(0)\to B_\eps(A)\subseteq\GLp(n)$ is a diffeomorphism;
	\item[ii)] for every $B\in B_\eps(A)$ there exists a length minimizer connecting $A$ and $B$, which is uniquely determined up to reparameterization;
	\item[iii)] $\dg(A,\vPhiA(M)) = \fullnorm{M}$ if $\fullnorm{M}\leq\eps$;
	\item[iv)] $\vPhiA(S_\eps(0))=S_\eps(A)$,
	\end{itemize}
	where
	\[
		S_\eps(0)=\{M\in\gln \setvert \fullnorm{M}=\eps\}
	\]
	and
	\[\pushQED{\qed}
		S_\eps(A)=\{B\in\GLp(n) \setvert \dg(A,B)=\eps\}\,.\qedhere
	\]\popQED
\end{lemma}
\begin{corollary}
\label{corollary:metricProperty}
	The geodesic distance $\dg$ is a \emph{metric} on $\GLpn$.\qed
\end{corollary}

\subsection{Global existence of minimizers}

After these preparations, we can now prove the global existence of length minimizing curves. We will closely follow the proof of the Hopf-Rinow theorem as given in \cite[Theorem 1.4.8]{Jost1998}.

\begin{proposition}
	\label{prop:hopfRinow}
	Let $A, B \in \GLp(n)$. Then there exists a length minimizing curve connecting $A$ and $B$, i.e.\ $X\in\adset{A}{B}$ with $\len(X) = \dg(A,B)$.
\end{proposition}
\begin{proof}
	Let $B\in\GLp(n)$ be fixed. For $A\in\GLp(n)$, we define $r_A\colonequals\dg(A,B)$ and construct a length minimizer $X_A\colon[0,r_A]\rightarrow \GLp(n)$ as follows:
	Choose $\eps(A)$ as in Lemma \ref{lemma:geodesicSpheres} and $\eps$ with $0<\eps<\eps(A)$, $\eps < \dg(A,B)$. Then the continuous function
	\begin{align}
		\gln\supset S_\eps(0)\rightarrow \R^+, \:N\mapsto\dg(\vPhiA(N),B)
	\end{align}
	attains a minimum at some $M_A\in S_\eps(0)$. Note that by this definition, $\vPhiA(M_A)$ is closest to $B$ among the geodesic sphere $S_\eps(A)$. Next we define $X_A\in C^\infty([0,r_A];\GLp(n))$ by
	\begin{align}
		X_A(t) \colonequals \vPhiA(t\,\eps^{-1}M_A)\,.
	\end{align}
	To simplify notation, denote by $\len_t(X_A) \colonequals \len(\eval{X_A}{[0,t]})$ the length of $X_A$ up to $t\in[0,r_A]$. Then equation \eqref{eq:geodesicLength} yields $\len_t(X_A) = t\,\eps^{-1}\fullnorm{M_A} = t$. Thus we find $\dg(X_A(t),A) \leq t$, and therefore
	\begin{align}
		\dg(X_A(t),B) &\geq \dg(A,B)-\dg(X_A(t),A) \geq \dg(A,B)-t\,.\label{eq:hopfRinowLowerBound}
	\end{align}
	We will denote by
	\begin{align}
		I_A\colonequals\{t\in[0,r_A] \setvert \dg(X_A(t),B) = r_A-t\}
	\end{align}
	the set of all $t$ where equality holds in \eqref{eq:hopfRinowLowerBound}. Geometrically, $I_A$ measures for how long $X_A$ runs \enquote{in an optimal direction towards $B$}. Since $\len(X_A) = \len_{r_A}(X_A) = r_A = \dg(A,B)$, $X_A$ is a length minimizer between $A$ and $B$ if $X_A(r_A) = B$, thus it remains to show that $r_A \in I_A$. We will first show that $\eps \in I_A$:
	
	Assume $\dg(X_A(\eps),B)>r_A-\eps$. Then $\dg(A,B)=r_A<\eps+\dg(X_A(\eps),B)$, hence we can find $Y\colon[a,b]\rightarrow\GLp(n)$, $Y\in\adsetAB$ with
	\begin{align}
		\len(Y)<\eps+\dg(X_A(\eps),B)=\eps+\dg(\vPhiA(M_A),B)\,.
	\end{align}
	Since $\dg(A,B)>\eps$ by definition of $\eps$, the curve $Y$ intersects $S_\eps(A)$, so there exist $s\in(a,b)$ and, according to Lemma \ref{lemma:geodesicSpheres}, $N\in S_\eps(0)\subset\gln$ with $Y(s)=\vPhiA(N)$. Since $\len(Y) = \len(\eval{Y}{[a,s]}) + \len(\eval{Y}{[s,b]})$, we obtain
	\begin{align}
		\dg(\vPhiA(N),B)&\leq \len(\eval{Y}{[s,b]}) = \len(Y) - \len(\eval{Y}{[a,s]}) \leq \len(Y) - \dg(Y(a),Y(s)) \nonumber \\
		&= \len(Y) - \dg(A,\vPhiA(N)) \;=\; \len(Y)-\eps < \dg(\vPhiA(M_A),B)\,,
	\end{align}
	in contradiction to the definition of $M_A$. Therefore $\eps\in I_A$, and in particular $I_A$ is nonempty.\\
	
\begin{figure}
	\begin{center}
		\includegraphics[width=1.0\textwidth]{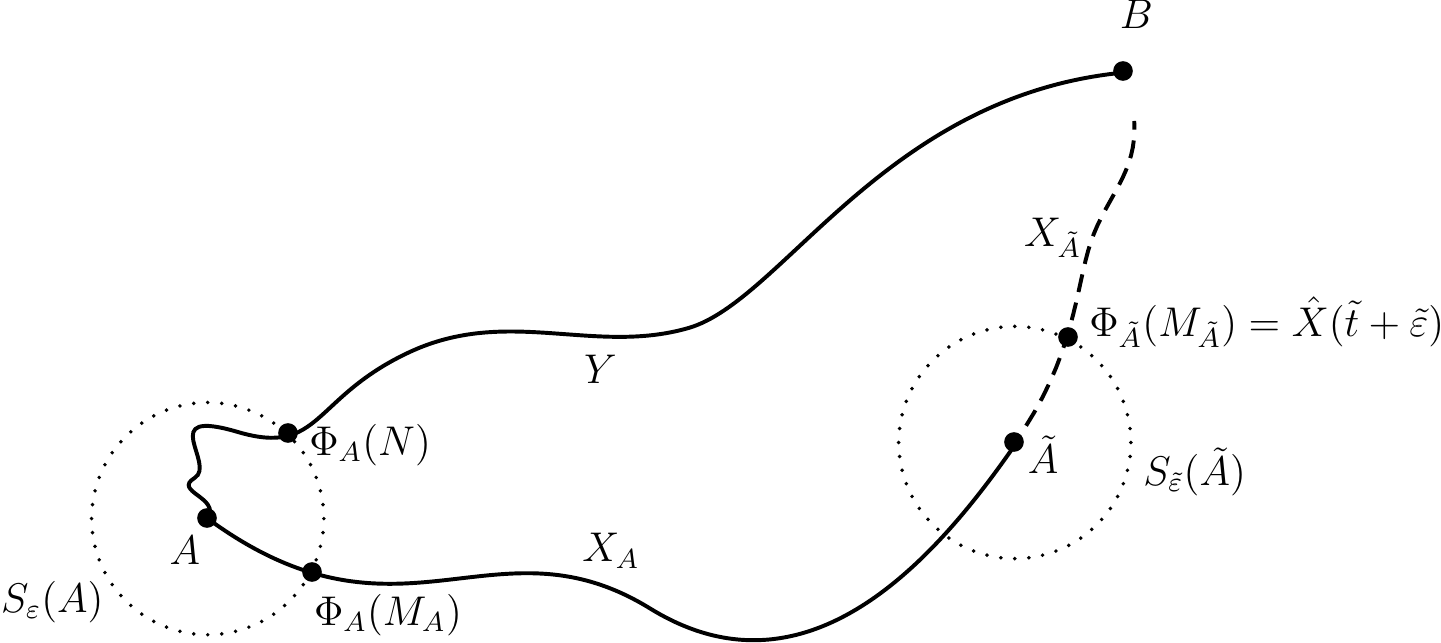}
		\caption{Global existence of length minimizers}
	\end{center}
\end{figure}
	
	Now assume $r_A \notin I_A$. It is not difficult to see that $I_A$ is closed, so let $\tilde{t} \colonequals \max I_A \neq r_A$ and, for $\tilde{A}\colonequals X_A(\tilde{t})$, choose $\tilde{\eps}<\eps(\tilde{A})$, $M_{\tilde{A}}$, $X_{\tilde{A}}$, $I_{\tilde{A}}$ accordingly. We find $\tilde{\eps}\in I_{\tilde{A}}$, and thus $\dg(X_{\tilde{A}}(\tilde\eps), B) = \dg(\tilde{A},B)-\tilde{\eps}$.\\
	Let $\hat{X}$ denote the piecewise smooth curve
	\begin{align}
		\hat X\colon [0,\tilde t + \tilde \eps] \rightarrow \GLp(n), \:\hat X(t) = 
			\begin{cases}
				X_A(t) &: t \leq \tilde t\\
				X_{\tilde A}(t-\tilde t) &: t > \tilde t
			\end{cases}\:\:
	\end{align}
	obtained by \enquote{attaching} $X_{\tilde A}$ to $X_A$. Then from
	\begin{align}
		\len(\hat X) &= \len(\left.\smash{\hat X}\vphantom{X}\right|_{[0, \tilde t]}) + \len(\left.\smash{\hat X}\vphantom{X}\right|_{[\tilde t, \tilde t + \tilde \eps]}) = \tilde t + \tilde \eps
	\end{align}
	and
	\begin{align*}
		\quad\,\dg(A, \hat X(\tilde t + \tilde \eps)) = \dg(A,X_{\tilde A}(\tilde\eps)) &\geq \dg(A,B) - \dg(X_{\tilde A}(\tilde \eps),B) \nonumber \\
		&\overset{\mathclap{\tilde\eps\in I_{\tilde A}}}{=} \dg(A,B) - \dg(\tilde A,B) + \tilde \eps \nonumber \\
		&= \dg(A,B) - \dg(X_A(\tilde t),B) + \tilde \eps \overset{\tilde t \in I_A}{=} \:\tilde t + \tilde \eps\,,
	\end{align*}
	we infer that $\hat X$ is a length minimizer between $A$ and $\hat X(\tilde t + \tilde\eps)$. Since $X_A$, $X_{\tilde A}$ have (the same) constant speed, $\hat X$ is a piecewise smooth energy minimizer in $\adset{A}{\hat X(\tilde t + \tilde\eps)}$. Then according to Proposition \ref{proposition:odeForPiecewiseSmoothEnergyMinimizers}, $\hat X$ is smooth and satisfies \eqref{eq:generalDifferentialEquation} everywhere. Finally, because $X_A$ satisfies the differential equation with the same initial values as $\hat X$, the uniqueness from Proposition \ref{prop:generalIVP} yields
	\begin{align}
		\hat X = X_A\text{ on }[0, \tilde t + \tilde \eps]\,,
	\end{align}
	and therefore
	\begin{align}
		\quad\:\dg(X_A(\tilde t + \tilde \eps),B) =\dg(X_{\tilde A}(\tilde\eps),B) &= \dg(\tilde A,B) - \tilde\eps \nonumber \\
		&= \dg(X_A(\tilde t),B) - \tilde\eps \nonumber \\
		&= \dg(A,B)-(\tilde t + \tilde \eps) \:\: \Rightarrow \:\: \tilde t + \tilde \eps \in I_A\,,
	\end{align}
	in contradiction to the choice of $\tilde t$. Therefore $r_A\in I_A$, and thus $\dg(X_A(r_A),B) = r_A-r_A=0$. Since $\dg$ is a metric on $\GLpn$ according to Corollary \ref{corollary:metricProperty}, we find $X(r_A)=B$. Since $\len(X_A)=r_A$, this concludes the proof.
\end{proof}

\subsection{Conclusion}
As above, let $g$ be a left-invariant, right-$\On$-invariant Riemannian metric on $\GLn$ given by
\begin{align}
	g_A(M,N) &= \fullproduct{A\inv M, \, A\inv N}\,, \nnl
	\fullproduct{M,N} &= \mu \innerproduct{\dev\sym M,\dev\sym N} + \mu_c \innerproduct{\skew M,\skew N} + \frac{\kappa}{n} (\tr M) (\tr N)\,, \nnl
	\innerproduct{M,N} &= \tr(M^T N) \nonumber
\end{align}
for $A\in\GLn$, $M,N\in\gln$. Combining the global existence of length minimizing geodesics (Proposition \ref{prop:hopfRinow}) with the fact that a length minimizer can be reparameterized into an energy minimizer $X\colon[0,1]\to\GLn$ (Lemmas \ref{lemma:constantSpeed} and \ref{lemma:energyLengthMinimizerRelation}) and the representation formula of energy minimizers from Theorem \ref{theorem:parameterizationOfGeodesics}, we obtain our main result.

\begin{theorem}
	\label{theorem:existenceAndParameterizationOfMinimizers}
	Let $A,B\in\GLpn$. Then there exists $M\in\gln$ such that the curve $X\colon [0,1] \to \GLpn$ with
	\begin{equation}
		X(t) = A \, \exp(t(\sym M - \omega\skew M)) \, \exp(t(1+\omega) \skew M)
	\end{equation}
	is a shortest curve connecting $A$ and $B$, i.e.\ $X(0)=A$, $X(1)=B$ and
	\begin{equation}
		\fullnorm{M} = \len(X) = \dg(A,B) = \infY \, \len(Y)\,.
	\end{equation}
\end{theorem}

\begin{corollary}\label{cor:analyticalDistanceExpression}
	The geodesic distance between $A$ and $B$ is given by
	\begin{align}
		\dg(A,B) = \min\{&\fullnorm{M} \setvert M\in\gln \,, \label{eq:geodesicDistanceFormula}\\ &\: A \exp(\sym M - \omega\skew M) \, \exp((1+\omega) \skew M) = B \}\,. \nonumber
	\end{align}
	In particular, the set is non-empty for all $A,B\in\GLpn$.
\end{corollary}

\section{Special cases and applications}
For given $A,B\in\GLpn$ it is still quite difficult to compute the geodesic distance $\dg(A,B)$ using formula \eqref{eq:geodesicDistanceFormula}: There is no known closed form solution for finding an $M\in\gln$ with
\begin{equation}
	\exp(\sym M - \omega\skew M) \, \exp((1+\omega) \skew M) = A\inv B\,,
\end{equation}
let alone one that is minimal with regard to the norm $\fullnorm{\,.\,}$. We will therefore consider some easier special cases.%

\subsection{The geodesic distance on $\GLp(1)$}
In the one dimensional case, we can identify $\GLp(1)$ with $\R^+$ and the matrix multiplication with the usual multiplication on $\R$. The most general inner product on $\R\cong\gl(1)$ is given by
\begin{equation}
	\innerproduct{x,y}_\kappa = \kappa\,x\,y,\quad \kappa > 0\,,
\end{equation}
with the corresponding left-invariant Riemannian metric
\begin{equation}
	g_{\realPointOne}(x,y) = \innerproduct{p^{-1}x,p^{-1}y}_\kappa = \Biginnerproduct{\frac{x}{p},\frac{y}{p}}_\kappa = \kappa\, \frac{x\,y}{p^2}\,.
\end{equation}
The length and energy of a piecewise differentiable curve $X\in C^0([0,1];\R^+)$ are therefore
\begin{align}
	\len(X) &= \int\nolimits_0^1 \sqrt{g_{X(t)}(\Xdot(t),\Xdot(t))} \dt	\,= \int\nolimits_0^1 \sqrt{\frac{\Xdot(t) \cdot \Xdot(t)}{X(t)^2}} \dt = \kappa \int\nolimits_0^1 \dynabs{\frac{\Xdot(t)}{X(t)}} \dt\,, \nonumber \\
	\ener(X) &= \int\nolimits_0^1 g_{X(t)}(\Xdot(t),\Xdot(t)) \dt	\,= \kappa^2 \int\nolimits_0^1 \left(\frac{\Xdot(t)}{X(t)}\right)^2 \dt\,.
\end{align}
It is easy to see that, in order to minimize the length over all curves connecting $p,q\in\R^+$, we can assume that $X$ is strictly monotone. In this case, $X$ is uniquely determined by $p$ and $q$ up to a reparameterization.\\
We recall from Lemma \ref{lemma:energyLengthMinimizerRelation} that a curve $X\in\adset{p}{q}$ is an energy minimizer if and only if it is a length minimizer of constant speed. Since the length is invariant under reparameterization (Lemma \ref{lemma:lengthInvariance}), $\len$ is constant on the set of strictly monotone curves $X \in \adset{p}{q}$, and therefore any $X$ with
\begin{equation}
	\label{eq:energyMinimizersOnGLp1}
	g_{X}(\Xdot,\Xdot) \equiv \text{constant},\quad X(0) = p,\: X(1) = q
\end{equation}
is an energy minimizer. To solve \eqref{eq:energyMinimizersOnGLp1}, we define
\begin{equation}
	\label{eq:energyMinimizerParameterizationOnGLp1}
	X\colon [0,1] \mapsto \GLp(1), \quad X(t) = p \cdot \exp(t\ln(\tfrac{q}{p}))\,,
\end{equation}
and check
\begin{align}
	g_{X(t)}(\Xdot(t),\Xdot(t)) = \kappa\frac{\Xdot(t)^2}{X(t)^2} = \kappa\frac{p^2 \exp(t\ln(\tfrac{q}{p}))^2 \, \ln(\tfrac{q}{p})^2}{p^2 \exp(t\ln(\tfrac{q}{p}))^2} \equiv \kappa(\ln(\tfrac{q}{p}))^2\,,
\end{align}
as well as
\begin{align}
	X(0) = p\,,\quad X(1) = p\, \frac{q}{p} = q\,.
\end{align}
Thus $X$ is an energy minimizer, and its energy and length are given by
\begin{align}
	\underset{Y\in\adset{p}{q}}{\min} \ener(Y) &= \ener(X) = \kappa^2 \int\nolimits_0^1 \left(\frac{\Xdot(t)}{X(t)}\right)^2 \dt = \kappa^2\int\nolimits_0^1 \left(\ln(\tfrac{q}{p})\right)^2 \dt = \kappa^2\abs{\ln(\smash{\tfrac{q}{p}})}^2 = \kappa^2(\ln(q) - \ln(p))^2 \nonumber
\intertext{and}
	\underset{Y\in\adset{p}{q}}{\min} \len(Y) &= \len(X) = \kappa \int\nolimits_0^1 \dynabs{\frac{\Xdot(t)}{X(t)}} \dt = \kappa\int\nolimits_0^1 \abs{\ln(\tfrac{q}{p})} \dt = \kappa\,\abs{\ln(\smash{\tfrac{q}{p}})} = \kappa\,\abs{\ln(q) - \ln(p)}\,.
\end{align}
We conclude:
\begin{proposition}
The geodesic distance between $p, q \in \GLp(1) \cong \R^+$ is
\begin{align}
	\dg(p,q) = \kappa\,\abs{\ln(\smash{\tfrac{q}{p}})} = \kappa\,\abs{\ln(q) - \ln(p)}\,,
\end{align}
and a shortest geodesic connecting $p$ and $q$ is given by
\begin{align}
	X\colon [0,1] \to \GLp(1), \quad X(t) = p \cdot \exp(t\ln(\tfrac{q}{p}))\,.
\end{align}
\end{proposition}

\subsection{Normal matrices}
The following lemma states some properties of normal matrices and their relation to the matrix exponential.
A matrix $M\in\gln$ is called normal if $MM^T = M^TM$.

\begin{lemma}
\label{lemma:normalExponential}
	Let $A\in\GLpn$, $M\in\gln$. Then:\\
	\begin{tabular}{ll}
		i) &$M$ is normal if and only if $\sym M$ and $\skew M$ commute. \\
		ii) &If $M$ is normal, then $\exp(M)$ is normal. \\
		iii) &If $A$ is normal, then there exists a normal matrix $N\in\gln$ with $\exp(N)=A$.\\
	\end{tabular}
\end{lemma}
\begin{proof}
	i) can be shown by direct computation:
	\begin{alignat}{3}
		&&(\sym M)(\skew M) &= (\skew M)(\sym M) \nonumber \\
		& \Leftrightarrow & \half(M + M^T) \cdot \half(M - M^T) &= \half(M - M^T) \cdot \half(M + M^T) \nonumber \\
		& \Leftrightarrow & \quad \tel4(M^2 + M^TM - MM^T - (M^T)^2) &= \tel4(M^2 - M^TM + MM^T - (M^T)^2) \nonumber \\
		& \Leftrightarrow & M^TM - MM^T &= -M^TM + MM^T \nonumber \\
		& \Leftrightarrow & M^TM &= MM^T\,.
	\end{alignat}
	Furthermore, if $M$ and $M^T$ commute, then according to \eqref{eq:expHomomorphism}, $\exp(M)$ and $\exp(M^T) = \exp(M)^T$ commute as well, proving ii).
	A proof of iii) can be found in \cite[Proposition 11.2.8]{Bernstein2009}.
\end{proof}

\subsubsection{Geodesics with normal initial tangents}
Let $N \in \gln$ be a normal matrix. Then, according to Theorem \ref{theorem:parameterizationOfGeodesics}, the geodesic curve $X\colon [0,1] \to \GLn$ with
\begin{align}
	X(t) = \vPhi(tN) = \exp(t(\sym N - \omega\skew N)) \: \exp(t(1+\omega) \skew N)
\end{align}
is the unique solution to the geodesic equation
\begin{align}
	\left\{\begin{aligned}
			\tangent &= X^{-1}\dot X\,,\\
			\dot \tangent &= \frac{1+\frac{\mu_c}{\mu}}{2}(\tangent^T\tangent - \tangent\tangent^T)
		\end{aligned}
	\right.
\end{align}
with the initial values
\begin{align}
	X(0) = \id\,, \quad \Xdot(t) = N\,.
\end{align}
According to Lemma \ref{lemma:normalExponential}, $\sym N$ and $\skew N$ commute if $N$ is normal. But then $\lambda_1 \sym N$ and $\lambda_2 \skew N$ commute as well for all $\lambda_1, \lambda_2 \in \R$, and thus $\exp(\lambda_1 \sym N + \lambda_2 \skew N) = \exp(\lambda_1 \sym N) \exp(\lambda_2 \skew N)$ according to \eqref{eq:expCommutation}. This allows us to simplify $X$: we find
\begin{align}
	X(t) &= \exp(t\,\sym N - t\,\omega \skew N) \:\exp(t(1+\omega) \skew N) \nonumber \\
	&= \exp(t\,\sym N) \:\exp(- t\,\omega \skew N) \:\exp(t(1+\omega) \skew N) \nonumber \\
	&= \exp(t\,\sym N) \:\exp(- t\,\omega \skew N + t(1+\omega) \skew N) \nonumber \\
	&= \exp(t\,\sym N) \:\exp(t \skew N) = \exp(t\,\sym N + t \skew N) \;=\; \exp(tN)\,. \label{eq:normalTangentParameterization}
\end{align}
This representation of geodesics with normal initial tangents can also be found in \cite[Section 8.5.1]{Taubes2011}. The length of $X$ is
\begin{align}
	\len(X) &= \int\nolimits_0^1 \sqrt{\fullproduct{X^{-1}\Xdot, X^{-1}\Xdot}} \:\dt \nonumber \\
	&= \int\nolimits_0^1 \sqrt{\fullproduct{\exp(tN)^{-1}\,\exp(tN)\,N \,,\, \exp(tN)^{-1}\,\exp(tN)\,N}} \:\dt \nonumber \\
	&= \int\nolimits_0^1 \sqrt{\fullproduct{N,N}} \:\dt = \fullnorm{N} = \sqrt{\mu\norm{\dev\sym N}^2 + \mu_c\norm{\skew N}^2 + \frac{\kappa}{n}\tr(N)^2}\,, \label{eq:lengthOfNormalTangentGeodesic}
\end{align}
and $X(0) = \id$, $X(1) = \exp(N)$. In particular, $X(1)$ is normal according to Lemma \ref{lemma:normalExponential} as it is the exponential of a normal matrix.

\subsubsection{Geodesics connecting $\id$ with a normal matrix}

Now let $A\in\GLpn$ be normal. To find a geodesic $X$ connecting $\id$ and $A$, we need to find $M\in\gln$ solving
\begin{align}
	A = \exp(\sym M - \omega \skew M) \exp((1+\omega)\skew M) = X(1)\,.
\end{align}
But, again due to Lemma \ref{lemma:normalExponential}, a normal matrix $A$ has a (generally not uniquely determined) \emph{normal logarithm}, i.e.\ there exists a normal matrix $\Log A \in \gln$ such that $\exp(\Log A) = A$. According to \eqref{eq:normalTangentParameterization}, with $N=\Log A$, the geodesic $X$ with initial tangent $\Log A$ has the form
\[
	X(t) = \exp(t(\sym \Log A - \omega \skew \Log A)) \:\exp(t(1+\omega)\skew \Log A) = \exp(t\Log A)\,,
\]
and hence
\begin{equation*}
	X(1) = \exp(\Log A) = A\,.
\end{equation*}
Thus for normal matrices $A$, the curve $X\col[0,1]\to\GLpn$ with $X(t)=\exp(t\Log A)$ is a geodesic connecting $\id$ and $A$, and \eqref{eq:lengthOfNormalTangentGeodesic} yields $\len(X) = \fullnorm{\Log A}$. We therefore obtain the upper bound
\begin{equation}
	\dg(\id,A) \le \len(X) = \fullnorm{\Log A}
\end{equation}
for the distance of a normal matrix $A\in\GLpn$ to the identity $\id$.

Note carefully that this does not immediately imply $\dg(\id,A) = \fullnorm{\Log A}$ : While $\dg(\id,A)$ is indeed the length of \emph{a} geodesic curve connecting $\id$ and $A$, such a geodesic is generally not uniquely determined, and it is therefore possible that $X$ is not the \emph{shortest} such geodesic. However, as shown in Lemma \ref{lemma:geodesicSpheres}, geodesic curves are \emph{locally} unique, which immediately implies the following proposition.
\begin{proposition}\label{prop:distanceCloseToIdentity}
	Let $\mu,\mu_c,\kappa>0$. Then there exists $\eps>0$ such that for every normal $A\in\GLpn$ with $\norm{A-\id}<\eps$, the geodesic distance between $\id$ and $A$ is given by
	\[
		\dg(\id,A) = \min \fullnorm{\Log A} \colonequals \min \{\fullnorm{M} \setvert \exp(M)=A\}\,.
	\]
\end{proposition}
In particular, for $F\in\GLpn$ with sufficiently small $\norm{F-\id}$, Proposition \ref{prop:distanceCloseToIdentity} can be applied to the positive definite symmetric (and therefore normal) matrix $F^TF$ to obtain the distance
\begin{align}
	\dist^2(F,F^{-T}) = \dist^2(F^TF,\id) &= \min \{\fullnorm{M}^2 \setvert \exp(M)=F^TF\}\nnl
	&= \mu\,\norm{\dev\log(F^TF)}^2 + \frac{\kappa}{n}\, [\tr\log(F^TF)]^2\,,\label{eq:localDistanceIdToU}
\end{align}
where $\log (F^TF)$ is the symmetric \emph{principal matrix logarithm} of $F^TF$ on $\PSymn$ \cite{Bernstein2009,higham2008}. Note, again, that it is not obvious at this point that equality \eqref{eq:localDistanceIdToU} holds for all $F\in\GLpn$ since it is not immediately clear that there is no shorter geodesic curve connecting $\id$ and $F^TF$.

\subsection{Application to nonlinear elasticity}
An example for the application of geodesic distance measures is the theory of nonlinear elasticity (and, more specifically, hyperelasticity), where the deformation of a solid body is modelled via an energy functional $W$ which depends on the gradient $F\in\GLpn$ of the deformation (see e.g.\ \cite[Chapter 4]{Ciarlet1988} for an introduction to hyperelasticity). Similarly, Mielke's work on the geodesic distance in $\SLn$ was primarily motivated by elasto-plastic applications \cite{Mielke2002}.

Among the energy functions considered in nonlinear elasticity is the \emph{isotropic Hencky strain energy}
\begin{equation}
	W\colon \GLpn\to\R, \quad W(F) = \mu \, \norm{\dev\log \sqrt{F^TF}}^2+\frac{\kappa}{2}\,[\tr(\log \sqrt{F^TF})]^2\,,
\end{equation}
which was introduced in 1929 by Heinrich Hencky \cite{Hencky1929}. Note that, because $\log\sqrt{F^TF}$ is symmetric, $W$ can be written as
\begin{equation}
	W(F) = \fullnorm{\log\sqrt{F^TF}}
\end{equation}
for arbitrary $\muc>0$. As was shown by Neff et al.\ \cite{Neff_Eidel_Osterbrink_2013,neff2015henckymain}, the Hencky energy can be characterized as the geodesic distance (with respect to a left-$\GLn$-invariant, right $\On$-invariant Riemannian metric) of the deformation gradient $F$ to the group $\SOn$ of rigid rotations. The proof of this result employs the parametrization of geodesic curves given in Theorem \ref{theorem:parameterizationOfGeodesics} and the characterization of the geodesic distance stated in Corollary \ref{cor:analyticalDistanceExpression} as well as a recently discovered optimality result regarding the matrix logarithm \cite{Lankeit2014,Neff_Nagatsukasa_logpolar13}.

\begin{proposition}
\label{prop:hencky}
	Let $g$ be the left-invariant Riemannian metric on $\GL(n)$ induced by the isotropic inner product $\fullproduct{\cdot,\cdot}$ on $\gln$ with $\muc\geq0$, and let $F\in\GLp(n)$. Then
	\begin{align}
		\dg(F,\SO(n)) = \dg(F,\,R) = \fullnorm{\log \sqrt{F^TF}} = \fullnorm{\log U}\,,
	\end{align}
	where $F=R\,U$ with $R\in\SO(n)$, $U=\sqrt{F^TF}\in\PSym(n)$ is the polar decomposition of $F$ and
	\begin{equation}
		\dg(F,\SO(n)) = \underset{Q\in\SO(n)}{\inf} \! \dg(F,Q)
	\end{equation}
	denotes the geodesic distance of $F$ to $\SOn$.
\end{proposition}
\begin{proof}
	See \cite{neff2015henckymain}.
\end{proof}
Proposition \ref{prop:hencky} also shows that equality \eqref{eq:localDistanceIdToU} holds globally, i.e.\ for all $F\in\GLpn$ and arbitrarily large $\norm{F-\id}$.

\subsection{Open Problems}
Although the explicit parametrization of geodesic curves makes it possible to establish some lower bounds for the geodesic distance in certain special cases, there is no known general formula or algorithm to compute the distance between two elements $A,B$ of $\GLpn$. However, recent results indicate that it might be possible to compute the geodesic distance for a number of additional special cases, including the (non-local) distance between arbitrary $A,B\in\SOn$ regarded as elements of $\GLn$ \cite{neffMartin2015SO}. Note that although the geodesic distance on $\SOn$ with respect to the canonical bi-invariant metric is already well known \cite{Moakher2002}, the distance discussed here takes into account the length of connecting curves which do not lie completely in $\SOn$. Furthermore, it might be useful to obtain some basic geometric properties of $\GLn$ or $\SOn$ with the considered metrics (e.g.\ to explicitly compute the curvature tensors).

\section*{Acknowledgements}
We are grateful to Prof.\ Robert Bryant (Duke University) for his helpful remarks regarding geodesics on Lie groups and invariances of inner products on matrix spaces as well as to Prof.\ Alexander Mielke (Weierstra\ss-Institut, Berlin) for pertinent discussions on geodesics in $\GL(n)$.

{
\footnotesize
\printbibliography[heading=bibnumbered]
}

\newpage
\begin{appendix}
\section{Appendix}
{\footnotesize
\label{section:expAppendix}
\subsection{Basic properties of the matrix exponential}
A proof for the following elementary properties of the matrix exponential can be found in \cite{higham2008}.
\begin{lemma}
	\label{lemma:basicExpProperties}
	Let $B_i\in\gl(d_i)$, $i\in\{1,\dotsc,k\}$ be square block matrices of size $d_i\times d_i$, $\lambda \in \R$ and $M,N\in\gln$. Then:
	\begin{alignat}{2}
		&\text{i)}& \det(\exp(M)) &= e^{\tr M}\label{eq:expDetTrace}\,,\\
		&\text{ii)}& \exp(\tfrac{\lambda}{n} \id) &= e^{\lambda}\id\,,\\
		&\text{iii)}& \exp(M^T) &= \exp(M)^T\,,\\
		&\text{iv)}& \exp(-M) &= \exp(M)^{-1}\,,\\
		&\text{v)}& MN=NM &\Rightarrow \exp(M+N) = \exp(M)\exp(N) = \exp(N)\exp(M)\label{eq:expHomomorphism}\,,\\
		&\text{vi)}& MN=NM &\Rightarrow M\exp(N) = \exp(N)M\label{eq:expCommutation}\,,\\
		&\text{vii)}& T\in\GLn &\Rightarrow \exp(T^{-1}MT) = T^{-1}\exp(M)T\label{eq:expTransform}\,,\\
		&\text{viii)}& M\in\so(n) &\Rightarrow \exp(M) \in \SO(n)\,,\label{eq:expsoSO}\\
		&\text{ix)}& M\in\Sym(n) &\Rightarrow \exp(M) \in \PSym(n)\,,\label{eq:expsymPSym}\\
		&\text{x)} &&\mkern-72mu\exp\begin{pmatrix}B_1&&0\\&\ddots\\0&&B_k\end{pmatrix} = \begin{pmatrix}\exp B_1&&0\\&\ddots\\0&&\exp B_i\end{pmatrix}\label{eq:expBlock}\,,\\
		&\text{xi)}& \Dexp{M}{T} &= \int\nolimits_0^1 \exp(sM)\: T\, \exp((1-s)M)\,\ds\,.\label{eq:expDerivative}
	\end{alignat}
	If $M$ and $T$ commute, formula xi) simplifies to
	\begin{equation}
		\Dexp{M}{T} = \exp(M)T = T\exp(M)\,.\label{eq:expSimpleDerivative}
	\end{equation}
\end{lemma}

\begin{proposition}
\label{prop:principalLogDiffeomorphism}
	The function
	\[
		\exp\colon \Sym(n) \to \PSym(n)
	\]
	is a diffeomorphism from $\Sym(n)$ to its open subset $\PSym(n)$. Its inverse
	\[
		\eval{\exp}{\Symn}\inv\colon \: \PSymn\to\Symn\,,
	\]
	is called the \emph{principal logarithm} on $\PSymn$ and is denoted by $\log$.
\end{proposition}
\begin{proof}
	We refer to \cite[Proposition 11.4.5]{Bernstein2009} for a proof that $\exp\colon \Sym(n) \to \PSym(n)$ is indeed injective and infinitely differentiable. To see that its inverse is differentiable, we refer to \cite[Theorem 6.6.14]{horn1994topics}, where it is shown that a primary matrix function $F$ defined through a real valued function $f$ acting on its eigenvalues is differentiable if $f$ is smooth on the set of eigenvalues attained on the domain of $F$. Since the principal logarithm on $\PSymn$ is such a function defined through $f=\ln$, and because $\ln$ is smooth on $\R^+$, i.e.\ on the set of eigenvalues attained on $\PSymn$, it is differentiable. For further information on the matrix exponential, the matrix logarithm and matrix functions in general we refer to \cite{higham2008}.
\end{proof}

\subsection{Additional proofs}

\begin{lemma}
\label{lemma:constantSpeedAppendix}
	Let $X \in \adm^1([a,b])$. Then there exists a unique piecewise differentiable $\varphi\in C^0([a,b];[a,b])$, $\varphi(a)=a$, $\varphi(b)=b$, $\varphi'(t)>0$ such that $X \circ \varphi$ has constant speed.
\end{lemma}
\begin{proof}
	Choose $a = a_0 < \dots < a_{m+1} = b$ such that $X$ is continuously differentiable on $[a,b] \setminus \{a_0,\dotsc,a_{m+1}\}$.
	Then the function
	\begin{equation}
		l\colon[a,b]\to[0,\len(X)], \quad l(t)\colonequals\int\nolimits_a^t \norm{\Xdot(t)}_{X(t)} \:\dt = \len(\eval{X}{[a,t]})
	\end{equation}
	is continuously differentiable on $[a,b] \setminus \{a_0,\dotsc,a_{m+1}\}$ as well with $l'(t) = \norm{\Xdot(t)}_{X(t)}$. By definition of regular curves (Definition \ref{def:admissibleSets}), $\Xdot(t) \neq 0$ for all $t \in [a,b] \setminus \{a_0,\dotsc,a_{m+1}\}$ and, because of the positive definiteness of Riemannian metrics, $\norm{\Xdot(t)}_{X(t)} = \sqrt{\fullg{X(t)}(\Xdot(t),\Xdot(t))} > 0$. Thus $l$ is a bijection and
	\begin{align}
		\tilde\varphi\colon [0,\len(X)] \to [a,b], \quad \tilde\varphi(s) = l^{-1}(s)
	\end{align}
	is well-defined, continuous on $[0,\len(X)]$ and continuously differentiable on $[0,\len(X)] \setminus \{l(a_0),\dotsc,l(a_{m+1})\}$. Let $\Xtilde \colonequals X \circ \tilde\varphi$. Then
	\begin{align}
		\dot\Xtilde(t) &= \ddt (X \circ \tilde\varphi)(t) = \tilde\varphi'(t) \Xdot(\tilde\varphi(t)) = \frac{1}{l'(\tilde\varphi(t))}\Xdot(\tilde\varphi(t)) = \frac{\Xdot(\tilde\varphi(t))}{\norm{\Xdot(\tilde\varphi(t))}_{X(\tilde\varphi(t))}}
	\end{align}
	if $\Xtilde$ is differentiable at $t$ and thus $\norm{\dot\Xtilde(t)}_{\Xtilde(t)} \equiv 1$, hence the transformation
	\begin{equation}
		\varphi\colon [a,b]\to[a,b], \:\varphi(t) \colonequals \tilde\varphi\left((t-a)\frac{\len(\gamma)}{b-a}\right)
	\end{equation}
	has the desired properties.\\
	To see the uniqueness of $\varphi$, note that if both $X$ and $X \circ \hat\varphi$ have constant speed, then
	\begin{alignat}{2}
		\norm{\ddt(X \circ \hat\varphi)(t)}_{(X \circ \hat\varphi)(t)} \equiv c \;&\implies\; \norm{\hat\varphi'(t) \cdot \Xdot(\hat\varphi(t))}_{X(\hat\varphi(t))} & &\equiv c \nonumber \\
		&\implies\; \abs{\hat\varphi'(t)} \equiv \frac{c}{\norm{\Xdot(\hat\varphi(t))}_{X(\hat\varphi(t))}} & &\equiv \text{constant}
	\end{alignat}
	for $t \in [a,b] \setminus \{a_0,\dotsc,a_{m+1}\}$. Then the restrictions $\hat\varphi(a) = a$, $\hat\varphi(b) = b$ and $\hat\varphi' > 0$ only allow for $\hat\varphi \equiv 1$.\qedhere
\end{proof}

\begin{lemma}
	\label{lemma:lengthInvarianceAppendix}
	Let $X \in \adm([a,b])$, and let $\varphi \in  C^0([c,d];[a,b])$ be a piecewise continuously differentiable function with $\varphi(c)=a$, $\varphi(d)=b$ and $\varphi'(t)>0$. Then
	\begin{equation}
		\len(X \circ \varphi) = \len(X)\,.
	\end{equation}
\end{lemma}
\begin{proof}
	Choose $c=a_0<a_1<\dots<a_{m+1}=d$ such that $\varphi$ is differentiable on $[c,d] \setminus \{a_0, \dotsc , a_{m+1}\}$ with $\varphi' > 0$ and $X$ is differentiable on $[a,b] \setminus \{\varphi(a_0),\dotsc,\varphi(a_{m+1})\}$. Then
	\begin{align}
		\len(X\circ\varphi) &= \int\nolimits_c^d \norm{\ddt(X\circ\varphi)(t)}_{X\circ\varphi(t)} \:\dt = \sum_{i=0}^m \int\nolimits_{a_i}^{a_{i+1}} \norm{\varphi'(t)\:\Xdot(\varphi(t))}_{X\circ\varphi(t)} \:\dt \nonumber \\
		&= \sum_{i=0}^m \int\nolimits_{a_i}^{a_{i+1}} \abs{\varphi'(t)} \: \norm{\Xdot(\varphi(t))}_{X(\varphi(t))} \:\dt = \sum_{i=0}^m \int\nolimits_{a_i}^{a_{i+1}} \varphi'(t) \: \norm{\Xdot(\varphi(t))}_{X(\varphi(t))} \:\dt \nonumber \\
		&= \sum_{i=0}^m \int\nolimits_{\varphi(a_i)}^{\varphi(a_{i+1})} \norm{\Xdot(t)}_{X(t)} \:\dt = \int\nolimits_a^b \norm{\Xdot(t)}_{X(t)} \:\dt \;=\; \len(X)\,. \nonumber\qedhere
	\end{align}
\end{proof}

\begin{lemma}
\label{lemma:canonicalProductAppendix}
	Let $M,N \in \gln$ and $\innerproduct{\cdot,\cdot}$ denote the canonical inner product $\innerproduct{M,N}=\tr(M^TN)$ on $\gln$. Then:
	\begin{alignat}{2}
		&\text{i)} &&\innerproduct{Q^TMQ, Q^TNQ} = \innerproduct{M,N} \quad \text{ for all } Q \in \On\,, \\
		&\text{ii)} &&\innerproduct{M,\id} = \tr(M)\,, \\
		&\text{iii)}\quad &&\innerproduct{S, W} = 0 \quad \text{ for all } S \in \Symn,\: W \in \son\,,
	\end{alignat}
\end{lemma}
\begin{proof}
	\begin{align}
		&\text{i)}& \innerproduct{Q^TMQ, Q^TNQ} &= \tr((Q^TMQ)^T (Q^TNQ)) = \tr(Q^TM^TQ Q^TNQ)\nnl
		&&&= \tr(Q^TM^TNQ) = \tr(M^TNQQ^T) = \tr(M^TN) = \innerproduct{M,N}\,.\\
		&\text{ii)}& \innerproduct{M,\id} &= \tr(M^T\id) = \tr(M^T) = \tr(M)\,.\\
		&\text{iii)}&\innerproduct{S,W} &= \tr(S^TW) = \tr(SW) = \tr(WS)\\
		&&&= -\tr(W^TS)= -\innerproduct{W,S} = -\innerproduct{S,W} \;\Longrightarrow\; \innerproduct{S,W} = 0\,.\nonumber\qedhere
	\end{align}
\end{proof}
~\\
\begin{lemma}
\label{lemma:isotropicProductLemmaAppendix}
	Let $M,N \in \gln$ and $\mu,\muc,\kappa\geq0$. Then
	\begin{alignat}{2}
		&\text{i)} &&\fullproduct{Q^TMQ, Q^TNQ} = \fullproduct{M,N} \quad \text{ for all } Q \in \On\,, \\
		&\text{ii)} &&\fullproduct{M, N} = \mu \innerproduct{\dev\sym M, N} + \mu_c \innerproduct{\skew M, N} + \frac{\kappa}{n} (\tr M) (\tr N) \nonumber \\
		&&& \hphantom{\fullproduct{M, N}} = \innerproduct{\mu \dev\sym M + \mu_c \skew M + \frac{\kappa}{n} \tr (M) \cdot \id\,,\,N}\,, \\
		&\text{iii)} &&\innerproduct{M, N}_{1,1,1} = \innerproduct{M,N}\,,\\
		&\text{iv)}\quad &&\fullproduct{S, W} = 0 \quad \text{ for all } S \in \Symn,\: W \in \son\,,
	\end{alignat}
	where $\Symn$ and $\son$ denote the set of symmetric and skew symmetric matrices in $\Rnn$ respectively.
\end{lemma}
\begin{proof}
	We first show that for $M\in\gln$ the matrices $\dev\sym M, \skew M$ and $\id$ are pairwise perpendicular with respect to the canonical inner product $\innerproduct{\cdot,\cdot}$. Lemma \ref{lemma:canonicalProductAppendix} yields $\innerproduct{\dev\sym M, \skew M}=0$ (note that $\dev\sym M$ is symmetric) as well as
	\begin{align}
		&\innerproduct{\skew M,\id} = \tr((\skew M)^T\id) = -\tr(\skew M) = -\tr(\id^T(\skew M)) = -\innerproduct{\skew M,\id} \quad\Rightarrow\quad \innerproduct{\skew M,\id} = 0\,,\nnl
		&\innerproduct{\dev\sym M, \id} = \tr(\dev\sym M) = \tr (M - \frac{\tr M}{n}\id) = \tr M - \frac{\tr M}{n}\tr\id = 0\nonumber\,;
	\end{align}
	Thus
	\begin{align}
		&\hspace{-7mm}\innerproduct{\mu \dev\sym M + \mu_c \skew M + \frac{\kappa}{n} \tr (M) \cdot \id\,,\,N}\nnl
		&= \innerproduct{\mu \dev\sym M + \mu_c \skew M + \frac{\kappa}{n} \tr (M) \cdot \id\,,\,\dev\sym N + \skew N + \frac{\tr N}{n}\id}\nnl
		&= \innerproduct{\mu \dev\sym M,\,\dev\sym N} + \innerproduct{\mu \dev\sym M,\,\skew N} + \innerproduct{\mu \dev\sym M,\,\frac{\tr N}{n}\id}\nnl
		&\quad+ \innerproduct{\mu_c \skew M,\,\dev\sym N} + \innerproduct{\mu_c \skew M,\,\skew N} + \innerproduct{\mu_c \skew M,\,\frac{\tr N}{n}\id}\nnl
		&\quad+ \innerproduct{\frac{\kappa}{n} \tr (M) \cdot \id\,,\,\dev\sym N} + \innerproduct{\frac{\kappa}{n} \tr (M) \cdot \id\,,\,\skew N} + \innerproduct{\frac{\kappa}{n} \tr (M) \cdot \id\,,\,\frac{\tr N}{n}\id}\nnl
		&= \innerproduct{\mu \dev\sym M,\,\dev\sym N} + \innerproduct{\mu_c \skew M,\,\skew N} + \innerproduct{\frac{\kappa}{n} \tr (M) \cdot \id\,,\,\frac{\tr N}{n}\id}\nnl
		&= \mu\innerproduct{\dev\sym M,\,\dev\sym N} + \muc\innerproduct{\skew M,\,\skew N} + \frac{\kappa}{n^2}\,\tr (M) \tr (N)\, \iprod{\id,\id}\nnl
		&= \mu\innerproduct{\dev\sym M,\,\dev\sym N} + \muc\innerproduct{\skew M,\,\skew N} + \frac{\kappa}{n} \tr (M) \tr (N) \quad=\quad \fullproduct{M,N}\,,
	\end{align}
	which proves ii). Then iii) follows directly by
	\begin{equation}
		\innerproduct{M,N}_{1,1,1} = \innerproduct{\dev\sym M + \skew M + \frac{\tr M}{n}\id,\, N} = \innerproduct{M,N}\,,
	\end{equation}
	and iv) from Lemma \ref{lemma:canonicalProductAppendix} by
	\begin{equation}
		\fullproduct{S,W} = \fullproduct{W,S} = \innerproduct{\mu \dev\sym W + \mu_c \skew W + \frac{\kappa}{n} \tr (W) \cdot \id\,,\,S} = \muc\innerproduct{W,\,S} = 0\,.
	\end{equation}
	Finally, since
	\begin{align}
		\tr(Q^TMQ) &= \tr(MQQ^T) = \tr(M)\,,\\
		\sym(Q^TMQ) &= \half(Q^TMQ + (Q^TMQ)^T) = \half(Q^T(M+M^T)Q) = Q^T(\sym M)Q\,,\nnl
		\skew(Q^TMQ) &= \half(Q^TMQ - (Q^TMQ)^T) = \half(Q^T(M-M^T)Q) = Q^T(\skew M)Q\,,\nnl
		\dev\sym(Q^TMQ) &= \dev(Q^T(\sym M)Q) = Q^T(\sym M)Q - \frac{\tr(Q^T(\sym M)Q)}{n}\,\id\nnl
		&= Q^T(\sym M)Q - \frac{\tr(\sym M)}{n}\,\id = Q^T(\sym M - \frac{\tr(\sym M)}{n}\,\id)Q \;=\; Q^T(\dev\sym M)Q\,,\nonumber
	\end{align}
	i) follows from
	\begin{align}
		\fullproduct{Q^TMQ, Q^TNQ} &= \mu \innerproduct{\dev\sym (Q^TMQ),\dev\sym (Q^TNQ)} + \mu_c \innerproduct{\skew (Q^TMQ),\skew (Q^TNQ)}\nnl
		&\quad+ \frac{\kappa}{n} \tr (Q^TMQ) \tr (Q^TNQ)\nnl
		&= \mu \innerproduct{Q^T(\dev\sym M)Q,\, Q^T(\dev\sym N)Q} + \mu_c \innerproduct{Q^T(\skew M)Q,\, Q^T(\skew N)Q}\nnl
		&\quad+ \frac{\kappa}{n} (\tr M) (\tr N)\nnl
		&= \mu \innerproduct{\dev\sym M,\dev\sym N} + \mu_c \innerproduct{\skew M,\skew N} + \frac{\kappa}{n} (\tr M) (\tr N) \;=\; \fullproduct{M,N}\,.\nonumber\qedhere
	\end{align}
\end{proof}

}
\end{appendix}

\end{document}